\documentclass[a4paper, 11pt]{amsart}
\pdfoutput=1 %

\usepackage{microtype}
\usepackage[margin=2.5cm]{geometry}

\usepackage{tikz}
\usetikzlibrary{perspective, positioning, cd, decorations.pathmorphing, calc, intersections, backgrounds, fit}

\usepackage{enumitem}

\usepackage{caption}
\usepackage{wrapfig}

\usepackage{bm}
\usepackage{centernot}
\usepackage{dsfont}
\usepackage{xparse}
\renewcommand{\mathbb}[1]{\mathds{#1}}

\usepackage{sansmath}

\newcommand{\ZZ}{\mathbb Z}
\newcommand{\QQ}{\mathbb Q}

\NewDocumentCommand{\FF}{e{_}}{\mathbb F\IfValueT{#1}{_{\mkern-3mu#1}}}
\newcommand{\card}[1]{\lvert#1\rvert}
\newcommand{\ol}[1]{\overline{#1}}
\NewDocumentCommand{\Laurent}{mm}{#1(\mkern-3mu(#2)\mkern-3mu)}
\NewDocumentCommand{\Affine}{e{_}}{\mathbb A_{#1}}

\newcommand{\Ga}{G_{\mathrm a}}
\newcommand{\Gm}{G_{\mathrm m}}

\DeclareMathOperator{\rk}{rk}
\DeclareMathOperator{\trdeg}{trdeg}
\DeclareMathOperator{\RM}{RM}

\DeclareMathOperator{\acl}{acl}
\DeclareMathOperator{\dcl}{dcl}
\DeclareMathOperator{\Cb}{Cb}
\DeclareMathOperator{\tp}{tp}

\newcommand{\Frob}{\mathord{\operatorname{Frob}}}
\DeclareMathOperator{\Aff}{Aff}

\DeclareMathOperator{\End}{End}
\DeclareMathOperator{\Aut}{Aut}

\DeclareMathOperator{\QAut}{QAut}

\DeclareMathOperator{\ACF}{\mathrm{ACF}}

\newcommand{\indep}[1][]{%
  \mathrel{\mathop{\vcenter{\hbox{\oalign{\noalign{\kern-.3ex}\hfil$\vert$\hfil\cr
    \noalign{\kern-.7ex}
    $\smile$\cr\noalign{\kern-.3ex}}}
  }}\displaylimits_{#1}}
}

\usepackage{mathtools}
\newcommand{\defas}{\coloneqq}
\newcommand{\asdef}{\eqqcolon}
\newcommand{\defiff}{\mathrel{\mathord:\!\!\mathord{\iff}}}
\newcommand{\Set}[1]{\left\{\,#1\,\right\}}
\newcommand{\isoto}{\xrightarrow{\smash{\raisebox{-0.5ex}{\ensuremath{\mathord\sim\mkern2mu}}}}}

\usepackage{romanbar}
\newcommand{\Rel}[1]{\mathrel{\text{\Romanbar{#1}}}}
\newcommand{\oRel}[1]{\mathord{\Rel #1}}

\usepackage{amsmath, amssymb, amsthm}
\usepackage{leftindex}

\usepackage[all]{xy}

\renewcommand{\restriction}{\mathord{\upharpoonright}}

\usepackage{comment}

\usepackage{hyperref}
\definecolor{darkred}{RGB}{160,0,0}
\definecolor{darkpurple}{RGB}{120,0,120}
\definecolor{lightpurple}{RGB}{140,50,150}
\definecolor{darkblue}{RGB}{0,0,160}
\hypersetup{
  colorlinks,
  citecolor=lightpurple,
  filecolor=black,
  linkcolor=darkblue,
  urlcolor=darkblue
}
\usepackage[sort,nameinlink]{cleveref}

\theoremstyle{definition}
\newtheorem{theorem}{Theorem}[section]
\newtheorem{lemma}[theorem]{Lemma}
\newtheorem{proposition}[theorem]{Proposition}
\newtheorem{claim}[theorem]{Claim}
\newtheorem{remark}[theorem]{Remark}
\newtheorem{corollary}[theorem]{Corollary}
\newtheorem{definition}[theorem]{Definition}

\newtheorem{notation}[theorem]{Notation}
\newtheorem{convention}[theorem]{Convention}

\usepackage{todonotes}

\title{Recognition of algebraic matroids is undecidable}

\author{Tobias Boege}
\address[Tobias Boege]{%
UiT The Arctic University of Norway, Tromsø, Norway}
\email{post@taboege.de}

\author{Geva Yashfe}
\address[Geva Yashfe]{University of Chicago}
\email{gyashfe@uchicago.edu}

\date{\today}



\begin{document}

\begin{abstract}
We prove that the recognition problem for algebraic matroids is undecidable. Explicitly, this means that there is no algorithm that takes as input a finite set $S$ and a function $r\colon\mathcal{P}(S) \to \mathbb{Z}_{\ge 0}$ (where $\mathcal{P}(S)$ is the power set) and decides whether there exists a pair of fields $F \subset K$, and a function $f\colon S \to K$, such that for all $A \subseteq S$: $\mathrm{trdeg}_F f(A) = r(A)$.

This problem is known to be decidable if the characteristic of the fields involved is constrained to be~zero. We prove that it is undecidable if the characteristic is either left unspecified (in which case a realization over any characteristic is accepted) or fixed to be a prime $p$.

The proof relies on Hrushovski--Zilber's Group Configuration Theorem and on the work of Evans and Hrushovski on ``Projective Planes in Algebraically Closed Fields''. We relate two different such projective planes, and eventually construct a reduction from the solvability of Diophantine equations over $\mathbb{F}_p(x)$ ($p$ prime) to algebraicity of matroids. Solvability of Diophantine equations over $\mathbb{F}_p(x)$ was proved to be undecidable by Pheidas for all $p > 2$, and later by Videla for $p=2$. A central part of our proof is a variant of the so-called Field Configuration Theorem.
\end{abstract}

\maketitle

\tableofcontents

\section{Introduction}

Transcendence bases of field extensions satisfy the Steinitz exchange property. This familiar property (from linear algebra and elsewhere) implies that all transcendence bases of the same field extension have the same cardinality, and that the transcendence degree function gives each field extension the structure of a matroid. Finite submatroids of such matroids are called \emph{algebraic}. Concretely, an algebraic matroid is a set $E$ and a \emph{rank function} $r\colon \mathcal{P}(E) \to \ZZ_{\ge 0}$ from the power set of $E$ to nonnegative integers, such that there exists a function $f\colon E \to K$, where $F \subset K$ are fields, satisfying
\[\forall A \subseteq E:\quad r(A) = \trdeg_F f(A).\]
(The rank function of a matroid is subject to additional constraints, called the rank axioms, but these are automatically satisfied when there exists $f$ as above.)

In this paper we prove that the recognition problem for algebraic matroids is undecidable. That is, there is no algorithm to decide whether a given matroid (which is the pair $(E,r)$ above) can be realized over some pair of fields. More precisely, we prove:
\begin{theorem} \leavevmode
\begin{enumerate}
    \item There is no algorithm which takes as input a matroid and decides whether it is algebraic in a given characteristic $p > 0$.
    \item There is no algorithm which takes as input a matroid and decides whether it is algebraic.
\end{enumerate}
\end{theorem}

Algebraic matroids have a very different flavor over characteristic zero as opposed to characteristic $p > 0$. In characteristic zero, a matroid is algebraic if and only if it is linear\footnote{That is, its elements can be mapped into a vector space in such a way that independence in the matroid corresponds exactly with linear independence in the vector space.}; this is seen essentially by passing to differentials (see \cite[Chp.~VIII, Prop.~5.5]{LangAlgebra}\footnote{The numbering is different in older editions. It is the section on derivations in the chapter on transcendental extensions.}). Many problems around linearity of matroids are decidable: linearity is decidable in a given characteristic, and so is the problem of linearity over \emph{some} characteristic, which is not fixed in advance; see~\cite[Sect.~6.8]{Oxley} for references. Algebraicity of matroids in characteristic $p > 0$ is not equivalent to linearity, and there exist matroids which are algebraic without being linear over any field. %

For further discussion of algebraic matroids in general, see \cite{Oxley,White}. For applications of algebraic matroids, see the survey \cite{AlgMatInAction}. Algebraic matroids are related to several realization problems in algebraic geometry, including the realizability of tropical varieties (see \cite{Yu17}), the realizability of algebraic cycles by irreducible varieties (\cite{Yu17} and \cite[Example 7.7]{HuhTropical}) and also the realizability problem for volume polynomials over a field (\cite[Proposition 5.4]{GHMSW} and \cite{CCLMZ}).

\subsection{Von Staudt constructions in algebraic matroids} 
The archetypal example of algebra arising from combinatorial geometry is the coordinatization of projective planes. It is well known that for every Desarguesian projective plane there exists a division ring $D$ coordinatizing the plane. Geometric configurations then force algebraic relations between the coordinates. This can be used to encode the solvability problem for systems of equations over $D$ in the realizability problem for point-line configurations within such a plane. An early use of this was made by von Staudt (\cite{vS57}); the idea is now sometimes referred to as ``the von Staudt construction''. Later, Mn\"ev \cite{Mnev} (and independently Sturmfels \cite{BokowskiSturmfels}) used it in a more matroid-theoretic context to prove a complexity-theoretic result and inspired many other applications of the same family of ideas. One source that explains the concept and basic construction in a way relevant to what we need is \cite[Sections~1.3 and~3]{SkewVonStaudt}.

\begin{figure}
\begin{tikzpicture}[every node/.style={fill, circle, thick, inner sep=3pt, outer sep=0pt}, every path/.style={thick}]
  \node [label={[yshift=-3ex]above:$[0:0:1]$}] (x) at (0,2) {};
  \node [label={left:$[1:0:0]$}] (z) at (-3,-2) {};
  \node [label={right:$[0:1:0]$}] (y) at (3,-2) {};
  \node [label={left:$[1:0:-1]$}] (a) at (-1.5,0) {};
  \node [label={[yshift=3ex]below:$[-1:1:0]$}] (b) at (0,-2) {};
  \node [label={right:$[0:-1:1]$}] (c) at (1.5,0) {};

  \draw (x) -- (a) -- (z) -- (b) -- (y) -- (c) -- (x);
  \draw (a) to[out=270, in=180] (b) to[out=0, in=270] (c);
\end{tikzpicture}
\caption{A picture of the group configuration~$M(K_4)$, realized in a projective plane over a skew field. Any algebraic realization of this matroid yields a one-dimensional algebraic group $G$ with an associated skew field~$L_F(G)$. This skew field coordinatizes a projective plane in which the six points lie as pictured.}
\label{fig: intro group conf}
\end{figure}
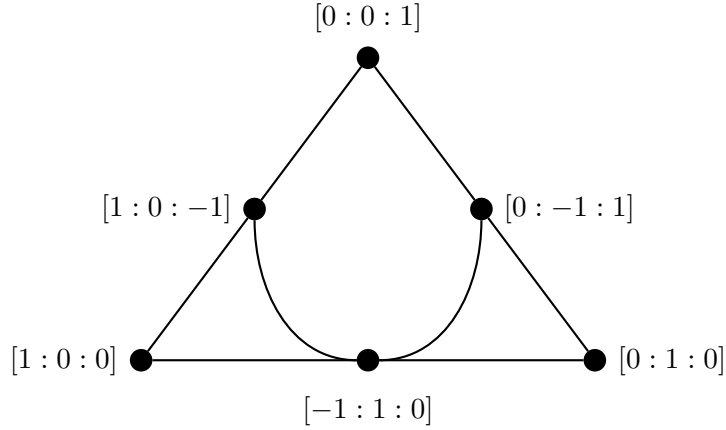

Evans and Hrushovski \cite{EH} showed how to use von~Staudt constructions in the context of algebraic matroids. Their idea is based on the Group Configuration Theorem, which essentially states that from any algebraic realization of the matroid $M(K_4)$ in a field extension $F \subset K$, one can read off a one-dimensional algebraic group $G$, uniquely defined up to isogeny. By adding additional points and lines to the figure one can encode polynomial equations over the Ore skew field of fractions of the ring of $F$-definable endomorphisms of $G$. We denote this skew field by $L_F(G)$.

The idea is that one can ensure that a given configuration is a subconfiguration of a projective plane over the skew field mentioned above, and then use von Staudt constructions. The original six points of the group configuration $M(K_4)$ embed in this projective plane as pictured in \Cref{fig: intro group conf}. A more detailed discussion of this is contained in \Cref{sec: EH planes}.

We need a variant of this construction that works in higher dimensional groups which are not necessarily commutative. The endomorphisms no longer form a ring (pointwise addition is impossible), but there is an analogue of the multiplicative group of the skew field. In \Cref{sec: quasi auts} we prove that points on the line between $[1:0:0]$ and $[0:1:0]$ correspond to what we call \emph{quasi-automorphisms} of $G$: these are subgroups of $G \times G$ that project onto each factor with finite kernel. Each of them is of the form
\[\{(g_1,g_2) \in G\times G \mid \varphi(g_1) = \psi(g_2)\}\]
for appropriately chosen nontrivial endomorphisms $\varphi,\psi$ of $G$. Such a quasi-automorphism corresponds to the point $[-\varphi: \psi: 0] = [-\psi^{-1} \varphi : 1 : 0]$ of the projective plane (see \Cref{rem: quasi-aut coordinates}). The quasi-automorphisms also exist for higher dimensional $G$.

\subsection{Overview of the proof}
If we force the group of a group configuration to be isogenous to a fixed $G$, then by adding points and lines to the corresponding matroid, we can relate the algebraicity problem to the solvability of equations in $L_F(G)$. There are two difficulties with this general approach to proving undecidability. First, it is not clear how to force $G$ to be isogenous to a particular group; and second, the existential theory of $L_F(G)$ is not known to be undecidable for any of the available groups. The skew fields one can obtain are:
\begin{itemize}[itemsep=0.4em]
    \item If $G=\Gm$ is the multiplicative group then $L_F(G) = \QQ$.
    \item If $G$ is an elliptic curve, $L_F(G)$ is either $\QQ$, an imaginary quadratic extension of $\QQ$, or a quaternion algebra over~$\QQ$.
    \item If $G=\Ga$ is the additive group, $L_F(G)$ is the skew field of fractions of the ring of $p$-polynomials. The ring of $p$-polynomials is isomorphic to Ore's ring of skew-polynomials $F[x;\Frob]$, whose elements are precisely the polynomials in $x$ over $F$, but in which the multiplication is twisted by the Frobenius: $x\cdot c = c^p\cdot x$ for $c \in F$. So $L_F(\Ga) = F(x;\Frob)$.
\end{itemize}

Our construction eventually produces a group configuration in which the group is $\Ga$ (this is straightforward) together with the extra point ${[-x: 1: 0]}$ in the projective plane over $F(x; \Frob) \simeq L_F(\Ga)$. By this we mean that we produce an algebraic matroid, and in every one of its algebraic realizations in fields $F \subset K$, the special point is coordinatized by ${[-x : 1: 0]}$ in the unique coordinate system that assigns the group configuration the coordinates ${[1:0:0]},{[0:1:0]},\ldots,{[0:-1:1]}$. (More precisely, this is true up to conjugation in $F(x;\Frob)$.) Producing this point is where most of the work lies. This gives us access to equations over $F(x;\Frob)$ with coefficients in $\ZZ[x]$. We can use this to encode systems of equations over $\FF_p(x)$, which is the centralizer of $x$ in $F(x;\Frob)$. This is accomplished by adding, for each variable $y_i$, the extra equation $xy_i = y_ix$ to the system. The solvability of systems of equations over $\FF_p(x)$ is undecidable by work of Pheidas ($p > 2$) and Videla ($p=2$): see the papers \cite{Pheidas,Videla}, or \Cref{thm: Pheidas--Videla} for our statement of their theorem. 

To construct the special point, we relate quasi-automorphisms of $\Ga$ and of $\Gm$ through the affine group $\Aff = \Gm \ltimes \Ga$. It turns out that in characteristic $p$, a quasi-automorphism of $\Aff$ inducing the $p$-th power map on $\Gm$ must induce the Frobenius on $\Ga$, up to conjugation in $L_F(\Ga)$. The Frobenius on $\Ga$ is represented by the element $x \in F(x;\Frob) = L_F(\Ga)$, so this accomplishes our goal. The $p$-th power map on $\Gm$ is just the element~$p$ within $L_F(\Gm)\simeq \QQ$. Writing $p = 1+1+\ldots+1$, von Staudt constructions can be used to produce this~point.

The remaining challenge is to construct a group configuration for $\Aff$ matroid-theoretically, and relate quasi-automorphisms of $\Ga$ and of $\Gm$ to those of $\Aff$. This is the heart of the paper. Relating between quasi-automorphisms of $\Aff$ and of $\Gm$ (via the quotient map) is accomplished in \Cref{sec: quasi auts}. We construct a group configuration for $\Aff$ by a new variant of the Field Configuration Theorem (utilizing a different configuration): the original does not seem to suffice, because its input includes a condition on canonical bases which is not stated in terms of ranks alone. This is accomplished in \Cref{sec: Aff}, where we also relate quasi-automorphisms of $\Aff$ with those of $\Ga$, essentially via the inclusion.

\section{Preliminaries}

We use various basic tools from model theory (see \cite{Marker,Palacin,Bays,Poizat}) and some terminology from matroid theory (\cite{Oxley}).

\subsection{Varieties and coordinates}
Algebraic varieties in this paper are defined and used in a manner close to \cite{Poizat}:

Let $\mathbb{K}$ be an algebraically closed field.
By an \emph{algebraic variety} over an algebraically closed subfield $F \subseteq \mathbb{K}$, we mean one given by a finite atlas of Zariski-closed charts $C_1,\ldots,C_n$ (where for each $i$, $C_i \subseteq \mathbb{K}^{d_i}$ is the zero set of a collection of polynomials with coefficients in $F$), with transition maps defined by rational maps with coefficients in~$F$. As a set, the variety is then the disjoint union of charts, modulo the equivalence relation generated by identifying points with their images under the transition maps. Since $\ACF_p$ has elimination of imaginaries, this set is definable over $F$. Note that this gives us access to notions such as the Morley rank of the point of a variety, or its algebraic closure. If $\mathbb{K}$ contains elements which are transcendental over $F$, an $F$-definable variety contains points which are not themselves $F$-definable whenever it is of dimension greater than~zero.

\begin{convention}\label[convention]{conv:variety_points}
    Let $V$ be an algebraic variety defined over a field $F \subseteq \mathbb{K}$, and let $P \in V$. By~$F(P)$ we mean the subfield of $\mathbb{K}$ generated by the union of $F$ with the coordinate tuple of $P$ in any one of the affine charts of $V$. Since transition maps are $F$-rational, this field is independent of the choice of chart.

    Note that $P \in \acl(F\cup\{P'\})$ for $P\in V$, $P'\in V'$ points in algebraic varieties definable over $F$ if and only if this holds for their coordinate tuples in some pair (and hence in all pairs) of affine charts.
\end{convention}

A \emph{definable} subset of a variety is the same thing as a \emph{constructible} subset. For definable subsets of varieties, Morley rank is equivalent to Krull dimension of the closure. We use a small amount of dimension theory: one relevant fact is that Morley rank is definable (\cite[6.2.20]{Marker}).
\subsection{Group configurations and algebraic groups}

Let $F$ be algebraically closed and $\mathbb{K}$ an algebraically closed extension. Assume $\mathbb{K}$ has at least countable transcendence degree over $F$. We work in the language of fields with the elements of $F$ adjoined as constants. So algebraic closure, rank, and so on are over~$F$.

\begin{figure}
\begin{tikzpicture}[every node/.style={fill, circle, thick, inner sep=3pt, outer sep=0pt}, every path/.style={thick}]
  \node [label={above:$x$}] (x) at (0,2) {};
  \node [label={left:$y$}] (y) at (-3,-2) {};
  \node [label={right:$z$}] (z) at (3,-2) {};
  \node [label={left:$a$}] (a) at (-1.5,0) {};
  \node [label={below:$b$}] (b) at (0,-2) {};
  \node [label={right:$c$}] (c) at (1.5,0) {};

  \draw (x) -- (a) -- (y) -- (b) -- (z) -- (c) -- (x);
  \draw (a) to[out=270, in=180] (b) to[out=0, in=270] (c);
\end{tikzpicture}
\caption{A group configuration.}
\label{fig: group conf}
\end{figure}
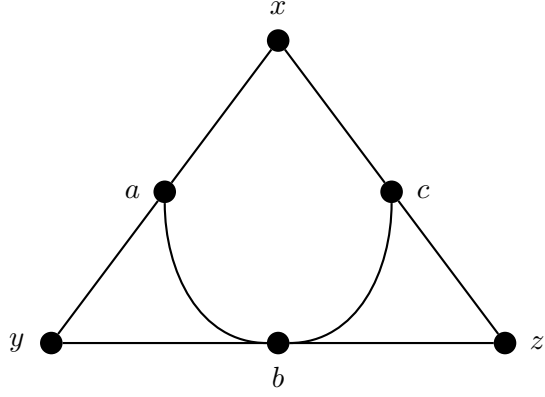

We state the version of the group configuration theorem that we need, then relate it to the more general version usually stated and proved in the literature.
\begin{definition}
A \emph{group configuration} in $\mathbb{K}$ over $F$ is a tuple $(a,b,c,x,y,z)$ of tuples in $\mathbb{K}$ such that:
\begin{enumerate}
\item Any non-collinear triple in \Cref{fig: group conf} is independent,
\item $\acl(a,b) = \acl(a,c) = \acl(b,c)$,
\item $\acl(a,x) = \acl(a,y) = \acl(x,y)$, and similarly for the triples $\{b,y,z\}$ and $\{c,x,z\}$.
\end{enumerate}
\end{definition}
See also \Cref{def: group config}.
\begin{theorem}
    Let $(a,b,c,x,y,z)$ be a group configuration over $F$ in $\mathbb{K}$. Then, possibly after extending $F$ by finitely many transcendentals of $\mathbb{K}$, independent of the given elements, there exists an $F$-definable group $G$ and a tuple $(a',b',c',x',y',z')$ of elements of $G$, such that each primed element is interalgebraic with the corresponding unprimed one, such that
    \[x'a' = y',\quad y'b' = z',\quad \text{and}\quad z'c' = x'.\]
\end{theorem}
So the input to the theorem is a group configuration, and the output is a tuple of elements in a group. It is convenient to also call the output tuples by the same term:
\begin{definition}
    \label{def: group config} A \emph{group configuration over an algebraic group $G$} is a tuple $(a,b,c,x,y,z)$ of elements of $G$ such that $x,y,z$ are independent generics, $xa = y$, $yb = z$, and $zc = x$.
\end{definition}

The form of the group configuration theorem that we cite takes slightly different input: instead of condition (3) in the definition of group configurations above, it assumes that $\acl(a,x) = \acl(a,y)$ and $\acl(\Cb(x,y/a)) = \acl(a)$ (and similarly for the triples $\{b,y,z\}$ and $\{c,x,z\}$). The following lemma bridges the gap.

\begin{lemma}
Suppose that tuples $a, x, y \in \mathbb{K}$ 
satisfy that
$\RM(a) = \RM(x) = \RM(y) = r$, any two are independent so have rank $2r$, and $\rk(a,x,y) = 2r$. Then $a$ is interalgebraic with~$\Cb(x,y/a)$.
\end{lemma}

\begin{proof}
In a stable theory like $\ACF_p$ types over models are stationary. Therefore the strong type $\tp(x,y/\acl(a))$ is stationary and it follows from \cite[Lem.~4.12]{Palacin} that $B = \Cb(x,y/a) \subseteq \acl(a)$. It remains to show that $a$ and $B$ have the same rank. The type $\tp(x,y/\acl(a))$ does not fork over its canonical base, so in particular $a \indep_{B} x,y$. In terms of ranks this means $\RM(a/B) + \RM(x,y/B) = \RM(a,x,y/B)$. The assumptions include the equality of ranks $\RM(x,y) = \RM(a,x,y)$ which implies $\RM(x,y/B) = \RM(a,x,y/B)$ and so $\RM(a/B) = 0$. %
\end{proof}

The following form of the group configuration theorem is the one most often proved in the literature. It applies much more generally in stable theories.
\begin{theorem}[Group configuration theorem]
\label[theorem]{thm: group config}
Let $(a,b,c,x,y,z)$ be a group configuration in $\mathbb{K}$. Then, after possibly expanding the language by parameters $B$ independent of $(a,b,c,x,y,z)$, there exists a connected right homogeneous space $(G, S)$, $\bigwedge$-definable over $\emptyset$, and a tuple $(a',b',c',x',y',z')$ such that each primed element is interalgebraic with the corresponding unprimed one, the elements of $(a',b',x')$ are pairwise independent, with $a',b'$ generic in $G$ and $x'$ generic in $S$, such that $b'=a'^{-1}\cdot_G c'^{-1}$, $y' = x'.a'$, and $z'=y'.b'$.
\end{theorem}
The parameters $B$ can be chosen to be independent of any other given subset of $\mathbb{K}$, so long as the transcendence degree of $\mathbb{K}$ is high enough. A proof can be found in \cite[Thm.~6.1]{Bays} (our multiplication conventions are flipped relative to this presentation). The same volume also contains a self-contained exposition of some model-theoretic preliminaries (\cite{Palacin}).

\begin{remark} \label{rem: algebraic group}
In $\ACF$, a $\bigwedge$-definable group is actually definable. This is due to Hrushovski~\cite{HrushovskiDiss} and follows from \cite[Cor.~5.19]{Poizat} and the fact that $\ACF_p$ is totally transcendental, meaning every definable set has finite Morley rank. A $\bigwedge$-definable homogeneous space is then also definable: it is a coset space for the stabilizer subgroup of an element, and this subgroup is definable. In~our applications of the theorem the stabilizers are actually finite.
Moreover, and also due to Hrushovski, every group definable in an algebraically closed field $F$ is definably isomorphic to an algebraic group over~$F$; see \cite{BouscarenWeil} for more context and a more algebraic proof. Poizat \cite[Sect.~4.5]{Poizat} remarks that this endows a definable group with a canonical Zariski topology.
\end{remark}

We use group configurations to get hold of groups. Sometimes we can then forget about part of the group configuration, and retain only three of the elements:
\begin{definition}
A \emph{group triple} in $\mathbb{K}$ over $F$ is a triple $(x,y,z)$ of tuples in $\mathbb{K}$ such that: 
\begin{enumerate}
\item There exists a connected group $G$, definable over $F$, for which
$x,y,z\in G$ and $x\cdot_{G}y=z$. 
\item Each element of the triple is generic in $G$, and any two of them
are algebraically independent over $F$, meaning:
\begin{gather*}
\RM(x/F)=\RM(y/F)=\RM(z/F)=\dim(G)\,\text{ and} \\
\RM(xy/F)=\RM(yz/F)=\RM(xz/F)=2\dim(G).
\end{gather*}
\end{enumerate}
Note that (1) and (2) together imply $\RM(xyz/F)=2\dim(G)$.
\end{definition}

In the notation of \Cref{thm: group config}, if the homogeneous space $(G,S)$ is the right action of $G$ on itself, each of
\[(x,a,y),\ (y,b,z),\ (z,c,x), \text{ and }(a,b,c)\]
is a group triple.
If $(G,S)$ is a homogeneous space, it is possible to upgrade it to a free right $G$-space, and replace the $6$-tuple of the configuration with a tuple of elements of $G$. This is standard, but short enough that we include it for completeness.

\begin{lemma}
    Let $(G,S)$ be a connected right homogeneous space for an $F$-definable group $G$, with $\dim S = \dim G$. Fix an $F$-definable $s_0 \in S$ and let $x \in S$. Then $x$ is interalgebraic over $F$ with each element of $\{g \in G \mid s_0.g = x\}$.
\end{lemma}
\begin{proof}
    Let $g\in G$ satisfy $s_0.g = x$. Then $g$ is one of finitely many elements of the set $\{g \in G \mid s_0.g = x\}$, hence algebraic over $x$ (algebraicity being considered over $F$): if this set is infinite then the stabilizer of $s_0$ is infinite, but then $\dim S = \dim G - \dim \mathrm{Stab}_G(s_0) < \dim G$.
    Conversely, $x$ is definable over $F(g)$.
\end{proof}
\begin{corollary}
    Let $(G,S)$ be a connected right homogeneous space for an $F$-definable group $G$ with $\dim S = \dim G$. Let $(a,b,c,x,y,z)$ satisfy $a,b,c\in G$, $x,y,z\in S$, $b = a^{-1} \cdot_G c^{-1}$, $x.a = y$, $y.b = z$, and $z.c = x$. Then there exist $\tilde{x},\tilde{y},\tilde{z} \in G$, interalgebraic with $x,y,z$ respectively, such that $\tilde x \cdot_G a = \tilde y$ and $\tilde y \cdot_G b = \tilde{z}$.
\end{corollary}
\begin{proof}
    Use the lemma above to find $\tilde{x} \in G$ interalgebraic with $x$. Then define $\tilde y = \tilde x \cdot_G a$ and $\tilde z = \tilde y \cdot_G b$.
\end{proof}

There is also a sense in which the group of a group configuration is unique. See \Cref{thm: quasi-iso} and \Cref{cor: quasi-iso config}.

\subsection{Definable and algebraic group homomorphisms} To say that a function between algebraic varieties is definable is to say that its graph is constructible. This implies that over some open subset of the domain, the graph is actually a Zariski-closed subset; but that does not mean that the restriction to this open is a morphism of varieties. It could be the composition of such a morphism with a negative power of the Frobenius (and this is essentially the only thing that can go wrong). 
We collect two lemmas for later use and refer to \cite{Poizat} or \cite{Marker} for further discussion.

\begin{lemma} \label{lemma: Def hom}
Let $G$ and $H$ be algebraic groups and let $\varphi\colon G \to H$ be a definable map which is a group homomorphism. Then there exists a $k \ge 0$ such that $\Frob^k \circ \varphi$ is a morphism of algebraic varieties.
\end{lemma}

The twist of $H$ by $\Frob^k$, which we denote by $\Frob^k(H)$, is an algebraic group satisfying that $\Frob^k\colon H \to \Frob^k(H)$ is an algebraic group homomorphism. So the lemma gives an algebraic group homomorphism $G \to \Frob^k(H)$. This argument is sketched in \cite{Poizat}; we include the details for completeness.

\begin{proof}
It suffices to prove that for each $x \in G$ there is an open neighborhood $x \in U \subset G$ and a $k \ge 0$ for which $\Frob^k \circ \varphi\restriction_U$ is a morphism: finitely many such $U$ cover $G$, and we can take the maximal $k$.

Let $C \subseteq \mathbb{K}^n$ be a Zariski-closed subset and consider a definable map $\psi\colon C \to \mathbb{K}$. By \cite[Proposition~3.2.14]{Marker} there exist finitely many sets $C_i \subseteq \mathbb{K}^n$, definable over $F$, such that $C = \bigcup_i C_i$ and $\psi \restriction_{C_i} = \Frob^{-k} \circ Q_i$ for a rational map~$Q_i$ and large enough~$k \ge 0$. Let $H = \bigcup_i Y_i$ be a finite cover by affine opens, so that for each $i$ the function $\varphi\restriction_{\varphi^{-1}(Y_i)}$ is definable from the open subset $\varphi^{-1}(Y_i)$ of $G$ to $Y_i \subset \mathbb{A}^n$ (for some $n$). Covering $\varphi^{-1}(Y_i)$ by finitely many affine opens and restricting to each in turn, the claim in \cite{Marker} gives a constructible cover of $\varphi^{-1}(Y_i)$ such that $\Frob^k \circ \varphi$ is rational on each set in the cover, (if $k$ is large enough,) because it is rational in each coordinate. Applying this to each $Y_i$, we get a finite cover $G = \bigcup_i C_i$ by constructible sets (each isomorphic to a constructible subset of an affine space) such that for each $i$ there exists $k_i$ for which $Q_i = \Frob^{k_i} \circ \varphi\restriction_{C_i}$ is rational (and defined on all of $C_i$).

As~$G^0$ is irreducible, some piece $C_j$ is dense in~$G^0$, which is open. Since $C_j$ is constructible, it contains an open subset $V$ of $G$, definable over $F$. By quasi-compactness, $G=\bigcup_{i=1}^n g_i V$ for some $g_1,\ldots,g_n \in G$. In fact, we can choose these $g_i$ in $V$: given $x \in G$, take any $z \in G$ generic over $x$. Then $g = xz$ is generic over $F$, hence in $V$, and also $z \in V$. So we have $x \in xzV = gV$. 

Denote $Q = Q_j \restriction_V$ and denote $H' = \Frob^k(H)$.
For each $z \in gV$, writing $z = g\cdot(g^{-1}z)$ we have 
\[
\Frob^k \circ \varphi(z) = Q(g)\cdot_{H'} Q(g^{-1}\cdot_G z).
\]
This is evidently a morphism.
\end{proof}

\begin{lemma}\label{lem:action_of_quotient}
    Let $G,E$ be algebraic groups and let $\pi\colon G \to E$ be a surjective homomorphism of algebraic groups. Let $f\colon G\times X \to X$ be an algebraic action of $G$ on a variety $X$, and assume that $f(g,x)$ is constant on each coset of $\ker\pi$. Then for each $x \in X$ there exists an integer $k\ge 0$ and a morphism of varieties $\overline{f_k}\colon E \to \Frob^k(X)$ such that $\overline{f_k} \circ \pi = \Frob^k \circ f(\cdot, x)$.
\end{lemma}
This is similar to \cite[p.~82, claim~7]{Poizat}.
\begin{proof}
    Fix $x \in X$ and denote by $\overline{f}\colon E \to X$ the (definable) function satisfying $\overline{f}(\pi(g)) = f(g,x)$ for all $g \in G$. As in \Cref{lemma: Def hom}, there exists an open subset $V \subseteq E$ and an integer $k>0$ such that $\Frob^k \circ \overline{f}\restriction_V$ is a morphism of algebraic varieties, and there is a finite open cover $E = \bigcup_i e_i V$ (with $e_i \in E$ for each $i$). For each $e_i$ choose a preimage $g_i \in G$. Now if $z \in e_i V$ then
    \[\overline{f}(z) = f(g_i, \overline{f}(e_i^{-1}z)).\]
    Denote by $f_k\colon G\times \Frob^k(X) \to \Frob^k(X)$ the action given by $f_k(g,x) = \Frob^k(f(g, \Frob^{-k}(x)))$, and observe that this action is algebraic. Therefore
    \[\Frob^k \circ \overline{f}(z) = \Frob^k \left( f(g_i, \overline{f}(e_i^{-1}z)) \right) =  f_k(g_i, \Frob^k \circ \overline{f}(e_i^{-1}z)).\]
    For all $z \in e_i V$ we have $e_i^{-1}z \in V$, and $\Frob^k \circ \overline{f}\restriction_V$ is a morphism of varieties. So by the equation above, $\Frob^k \circ \overline{f}\restriction_{e_i V}$ is a morphism of algebraic varieties for all $i$, hence $\Frob^k \circ \overline{f}$ is a morphism on all $E$.
\end{proof}

\subsection{Evans--Hrushovski planes}
\label{sec: EH planes}

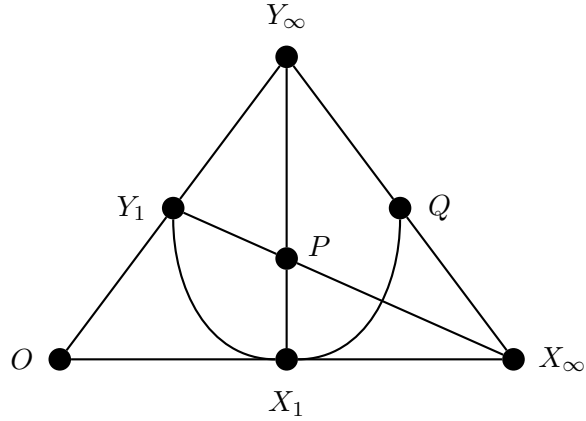
\begin{figure}
\begin{tikzpicture}[every node/.style={fill, circle, thick, inner sep=3pt, outer sep=0pt}, every path/.style={thick}]
  \node [label={[yshift=-0.4ex]above:$Y_\infty$}] (Yinf) at (0,2) {};
  \node [label={left:$O$}] (O) at (-3,-2) {};
  \node [label={right:$X_\infty$}] (Xinf) at (3,-2) {};
  \node [label={left:$Y_1$}] (Y1) at (-1.5,0) {};
  \node [label={below:$X_1$}] (X1) at (0,-2) {};
  \node [label={right:$Q$}] (Q) at (1.5,0) {};

  \node [label={[xshift=-0.4ex, yshift=1ex]right:$P$}] (P) at (intersection of  Y1--Xinf and X1--Yinf) {};

  \draw (X1) -- (P) -- (Yinf);
  \draw (Y1) -- (P) -- (Xinf);

  \draw (Yinf) -- (Y1) -- (O) -- (X1) -- (Xinf) -- (Q) -- (Yinf);
  \draw (Y1) to[out=270, in=180] (X1) to[out=0, in=270] (Q);
\end{tikzpicture}
\caption{Start with a projective basis $O, P, X_\infty, Y_\infty$ inside a projective plane. The intersection of the lines $Y_\infty \vee P$ and $O \vee X_\infty$ defines a point $X_1$ which determines a scale on the $X$-axis; the point $Y_1$ is constructed similarly. Finally, $Q$ is the intersection of $X_1 \vee Y_1$ and the line at infinity $X_\infty \vee Y_\infty$. Forgetting the point $P$ yields a group configuration.}
\label{fig: projective group conf}
\end{figure}

To an extension $\mathbb{K}/F$ of algebraically closed fields there is an associated combinatorial geometry $\mathcal{A}(\mathbb{K}/F)$ called the \emph{full algebraic matroid}. Its ground set is $\mathbb{K}$ and its closure operator sends any subset $A \subseteq \mathbb{K}$ to the algebraic closure $\acl(A) \defas \ol{F(A)} \subseteq \mathbb{K}$. The rank of a subset $A$ is the transcendence degree of $F(A)$ over $F$. Consider the flats of rank one as points, flats of rank two as lines, and so on. Evans and Hrushovski \cite{EH} showed that every set of points in $\mathcal{A}(\mathbb{K}/F)$ which satisfies the axioms of a projective plane can be coordinatized by a skew~field.
The key observation is that a projective basis in a projective plane yields a group configuration (cf.~\Cref{fig: projective group conf}). The associated one-dimensional group $G$ is commutative and the $F$-definable endomorphisms of $G$ form a~ring $\End_{F}(G)$. The Ore skew field of fractions of this ring coordinatizes the projective plane.
\begin{remark}
    In this paper, the elements of $\End_F(G)$ are $F$-definable endomorphisms which are not necessarily morphisms of algebraic varieties (for instance, if $G$ is defined over $\FF_p$, the inverse of the Frobenius is in this ring).

    We'll use the notation $\End_F^\mathrm{alg}(G)$ to denote those endomorphisms which are morphisms of algebraic varieties. This conflicts slightly with the usage in Evans--Hrushovski, but the conflict is not essential. The ring $\End_F(G)$ always contains $\End_F^\mathrm{alg}(G)$ and is contained in its Ore skew field of fractions. Hence the fields of fractions of these two rings are identical.
\end{remark}

\begin{theorem}[{Planar case of \cite[Theorem~3.3.1]{EH}}]
If $\Pi$ is the set of points of a projective plane in $\mathcal{A}(\mathbb{K}/F)$ then there exists a one-dimensional, connected algebraic group $(G, *)$ and independent generics $x, y, z \in G$ such that
\[
  \Pi \subseteq \Pi(G\colon x,y,z) \defas \{\, \acl(\alpha(x) * \beta(y) * \gamma(z)) \mid \text{$\alpha, \beta, \gamma \in \End_{F}(G)$ not all zero} \,\}.
\]
This plane is isomorphic to the projective plane over the skew field of fractions of $\End_F(G)$ via
\[\acl(\alpha(x) * \beta(y) * \gamma(z)) \mapsto [\alpha : \beta: \gamma].\]
\end{theorem}
\begin{remark}
    In the main body of the paper, we think of algebraic matroids slightly differently than this. For us, a point of an algebraic matroid will simply be a field element $x \in \mathbb{K}$, rather than the entire algebraically closed field $\overline{F(x)}$. The difference is that we can then have parallel points (we are no longer considering simple matroids). This is convenient because we do algebra with specific elements of fields and specific points on algebraic varieties; it would not be productive to identify interalgebraic elements, or to identify all $F$-definable points of all algebraic varieties with the trivial flat.
\end{remark}
We review the groups that can appear and their endomorphism rings, following \cite[Section~3]{EH} and its references. A one-dimensional connected algebraic group definable over an algebraically closed field~$F$ is either the additive group~$\Ga$, the multiplicative group~$\Gm$, or an elliptic curve.
Each of the possible groups~$G$ is commutative and hence $\End_{F}(G)$ is a ring with pointwise addition and where multiplication is composition of endomorphisms.
If $\varphi \in \End_{F}(G)$ is not zero then it is an endogeny (i.e. an endomorphism which is an isogeny): its image must have dimension one and since the image of a homomorphism is closed, $\varphi$ is surjective and has a finite kernel. It turns out that each possible $\End_{F}(G)$ satisfies the left Ore condition and therefore embeds into a skew field of fractions~$L_{F}(G)$. The~following skew fields arise in positive characteristic~$p$:
\begin{itemize}[itemsep=0.4em]
\item If $G \simeq \Ga$ then each element of $\End_F^\mathrm{alg}(G)$ is a $p$-polynomial, i.e. a polynomial function of the form $x \mapsto \sum_i c_i x^{p^i}$.
Thus $\End^\mathrm{alg}_F(\Ga)$ is isomorphic to the skew polynomial ring $F[x;\Frob]$: its elements are the polynomials with coefficients in $F$, and
a~skew polynomial $\sum_{i=0}^m c_i x^i$, with $c_i \in F$, is interpreted as the endomorphism sending $a \in \Ga$ to $\sum_{i=0}^m c_i \Frob^i(a)$. Here addition is addition of polynomials and multiplication is composition of endomorphisms. This ring is non-commutative: for $c \in F$, $x c = \Frob(c) x$.
We denote its Ore skew field of fractions by $L_F(\Ga) = F(x;\Frob)$.

\item If $G = \Gm$ then $\End_F^\mathrm{alg}(G) = \ZZ$ with (skew) field of fractions $L_{F}(\Gm) = \QQ$.

\item If $G$ is an elliptic curve, the resulting skew field may be $\QQ$, an imaginary quadratic extension $\QQ(\sqrt{d})$, where $d < 0$, %
or the quaternion algebra~$B_{p,\infty}$ which is ramified at exactly $p$~and~$\infty$. The endomorphism ring itself is an order in the respective $\QQ$-algebra.
\end{itemize}

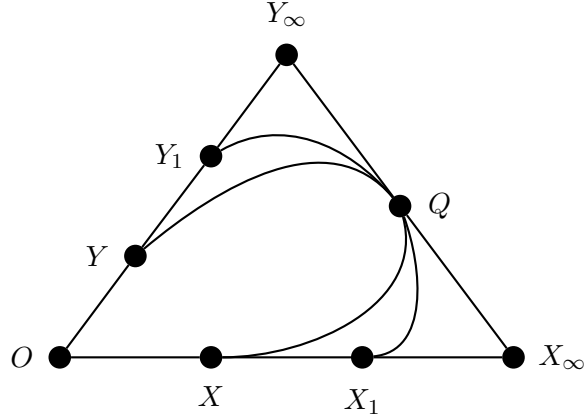
\begin{figure}
\begin{tikzpicture}[every node/.style={fill, circle, thick, inner sep=3pt, outer sep=0pt}, every path/.style={thick}]
  \node [label={[yshift=-0.4ex]above:$Y_\infty$}] (Yinf) at (0,2) {};
  \node [label={left:$O$}] (O) at (-3,-2) {};
  \node [label={right:$X_\infty$}] (Xinf) at (3,-2) {};
  \node [label={left:$Y_1$}] (Y1) at (-1,0.66) {};
  \node [label={left:$Y$}] (X') at (-2,-0.66) {};
  \node [label={below:$X_1$}] (X1) at (1,-2) {};
  \node [label={below:$X$}] (X) at (-1,-2) {};
  \node [label={right:$Q$}] (Q) at (1.5,0) {};

  \draw[thick] (Yinf) -- (Y1) -- (X') -- (O) -- (X) -- (X1) -- (Xinf) -- (Q) -- (Yinf);
  \draw[thick] (Y1) to[out=30, in=130] (Q) to[out=290, in=5] (X1);
  \draw[thick] (X') to[out=40, in=130] (Q) to[out=285, in=0] (X);
\end{tikzpicture}
\caption{Two superimposed group configurations $(X,Y,Q,O,X_\infty,Y_\infty)$ and~$(X_1,Y_1,Q,O,X_\infty,Y_\infty)$.}
\label{fig: endo labeling}
\end{figure}

It is important for us to be able to construct point-and-line configurations such that in any algebraic realization, some specified subset lies in a projective plane. Consider an algebraic realization of the configuration in \Cref{fig: projective group conf}. Since it contains a group configuration, there is an associated algebraic group $G$ such that the configuration is part of a maximal projective plane $\Pi = \Pi(G\colon g,h,k)$, where $g,h,k\in G$ are independent generics satisfying 
\[\acl(g) = O, \quad \acl(h) = X_\infty,\quad  \text{and}\quad \acl(k) = Y_\infty,\]
and such that $X_1 = \acl(g^{-1} h)$ i.e., $X_1$ has homogeneous coordinates $[-1 : 1 : 0]$ as a point in~$\Pi$.

Now let $X$ be some point in $\mathcal{A}(\mathbb{K}/F)$ which lies on the line~$\acl(O, X_\infty)$. In general, $X$ need not be part of~$\Pi$. But Evans--Hrushovski prove that $X \in \Pi$ if we assume, in addition, that the line $X \vee Q$ intersects $O \vee Y_\infty$ in some point~$Y$ (as in \Cref{fig: endo labeling}). This is related to our \Cref{thm: quasi-aut}, which also works for higher dimensional groups. This comes at the cost of replacing the coordinatizing skew field with a group (the group of quasi-automorphisms). See the discussion at the beginning of \Cref{sec: quasi-aut labels} and \Cref{rem: quasi-aut coordinates} for the correspondence between quasi-automorphisms and homogeneous coordinates in the one-dimensional case.

Thus, the matroid gadget of \Cref{fig: endo labeling} can be used to introduce variables from the skew field $L_{F}(G)$. By using von~Staudt constructions inside of a group configuration, any system $f_1(x) = \dots = f_n(x) = 0$ of Diophantine equations can be encoded in a matroid $M$ such~that: $M$~is~algebraic if and only if there is some one-dimensional, connected algebraic group~$G$ definable over~$\ol{\FF_p}$ such that $f_1(x) = \dots = f_n(x) = 0$ has a solution over~$L_{\ol{\FF_p}}(G)$.
See \cite[Section~3]{EH} and \cite{SkewVonStaudt} for further details about von~Staudt constructions.

The group configuration theorem by itself does not predict which of the possible groups underlies a given algebraic representation. By~using von~Staudt constructions, we can impose the solvability of certain Diophantine equations in~$L_{\ol{\FF_p}}(G)$. This mechanism gives some control over the groups and skew fields which appear.
Using only positive existential formulas one cannot hope to single out $\QQ$, for instance, because any equation with a solution in $\QQ$ would also have a solution in quadratic extensions and the quaternion algebra. Even so, it is not hard to show that recognizing algebraic matroids is at least as difficult as Hilbert's tenth problem over $\QQ$, that is, deciding whether a system of polynomial equations has a rational solution. This follows straightforwardly from \cite{EH} but does not seem to have been stated explicitly.

\begin{proposition} \label{prop: picking skew field}
Let $p > 0$ be a fixed prime. Solvability of Diophantine equations over $\QQ$ effectively reduces to recognizing algebraic matroids in characteristic $p$.
\end{proposition}
(It isn't known whether Hilbert's tenth problem over $\QQ$ is undecidable, so our work doesn't end here.)
\begin{proof}
$B_{p,\infty}$ is the only non-commutative skew field not of characteristic~$p$ obtainable as the Ore skew field of fractions of $\End_F(G)$, where $F$ is any field of characteristic $p$. Thus, all other skew fields can be ruled out via the Diophantine sentence $(\exists t_1: t_1 p = 1) \land (\exists x,y, t_2: t_2 (xy - yx) = 1)$. A theorem of Deuring (see~\cite[Chapter~42]{Voight}) ensures that $B_{p,\infty}$ really does appear as the skew field of a supersingular elliptic curve over $\ol{\FF_p}$ for any $p$.

Observe that $\QQ$ is the center of~$B_{p,\infty}$. Depending on the residue class of $p \bmod 4$ one can compute integers $a, b < 0$ such that whenever $i,j \in B_{p,\infty}$ satisfy the formula $i^2 = a \land j^2 = b \land ij = -ji$, then $\{1,i,j,ij\}$ is a $\QQ$-basis of $B_{p,\infty}$; see \cite[Example~14.2.13 \& Section~2.2]{Voight}. Then the center $\QQ$ is existentially defined as the set of solutions $x$ to $ix = xi \land jx = xj$. With this gadget, the solvability of Diophantine equations over $\QQ$ can be encoded in algebraicity of matroids over any characteristic~$p > 0$.
\end{proof}

\begin{remark}
The proof breaks down in characteristic $p = 0$. If $F$ is an algebraically closed extension of $\QQ$ then $L_{F}(\Ga) = F$ and the other skew fields which can arise from $\Gm$ and elliptic curves are contained in $\{\, \QQ \,\} \cup \{\, \QQ(\sqrt{d}) : d \in \ZZ, d < 0 \,\}$. Hence, they are all contained in~$F$. In~fact, it is known that a matroid is algebraic in characteristic zero if and only if it is linear over~$\mathbb{C}$, and this is decidable. %
\end{remark}

\section{Quasi-isomorphisms and quasi-epimorphisms}
\label[section]{sec: quasi auts}

The coordinatization result of Evans--Hrushovski relies on the fact that connected algebraic groups of rank one are commutative, so that their endomorphisms form a ring. Algebraic groups of rank greater than one need not be commutative and their endomorphisms (or endogenies), in general, need not be closed under pointwise ``addition''. What remains of the endomorphism ring is the multiplicative structure. In this section we develop the necessary generalizations of the Evans--Hrushovski machinery in terms of \emph{quasi-isomorphisms}~\cite{QuasiEndo}.

\subsection{The basic setup}
We define quasi-isomorphisms and quasi-epimorphisms. We then relate the group of quasi-automorphisms of a one-dimensional algebraic group $G$ to the skew field of fractions of its endomorphism ring, as described in \cite{EH}.

\begin{definition}
Let $G$ and $H$ be definable groups. A \emph{quasi-epimorphism} $G \to H$ is a definable subgroup $\Phi \le G \times H$ such that each of the projection maps $\pi_G\restriction_\Phi\colon \Phi \to G$ and $\pi_H\restriction_\Phi\colon \Phi \to H$ is surjective and $\pi_G\restriction_\Phi$ has finite kernel.
If,~additionally, $\pi_H\restriction_\Phi$ has finite kernel then $\Phi$ is a \emph{quasi-isomorphism}. A~quasi-isomorphism $G \to G$ is a \emph{quasi-automorphism} of~$G$.
\end{definition}

\begin{remark}\label[remark]{rem:quasi_iso_as_map}
A quasi-isomorphism is a finite-to-finite correspondence whose projection maps $\Phi \to G$ and $\Phi \to G'$ are isogenies. Therefore, two groups are quasi-isomorphic if and only if they are isogenous. Another way to view a quasi-isomorphism is as an ``isomorphism up to finite indeterminacy''.
If~$N,N'\triangleleft \Phi$ are the kernels of the projections to $G$ and $G'$ respectively, $\Phi$ induces a surjection
\[
  G \simeq \Phi/N\twoheadrightarrow \Phi/(N\cdot N') \simeq (\Phi/N')/(N/N\cap N') \simeq G'/\pi_{G'}(N)
\]
and similarly it induces a surjection $G'\twoheadrightarrow G/\pi_{G}(N')$.
These surjections induce mutually inverse isomorphisms $G/\pi_{G}(N')\simeq \Phi/(N\cdot N')\simeq G'/\pi_{G'}(N)$.

Conversely, if $K\triangleleft G$ and $K'\triangleleft G'$ are finite normal subgroups and we have an isomorphism $\varphi\colon G/K \isoto G'/K'$ then $\Phi = \Set{(g,g')\in G\times G' \mid \varphi(gK) = g'K'}$ is a quasi-isomorphism.
Note that the kernel of the coordinate projection $\Phi \rightarrow G$ is $N = \Phi\cap\left(\{e_{G}\}\times G'\right) = \{e_{G}\}\times K'$, and similarly the kernel of the projection $\Phi \rightarrow G'$ is $N' = K\times\{e_{G'}\}$.
\end{remark}

\begin{remark}
If $\Phi$ is a quasi-isomorphism, the induced map $G \to G'/\pi_{G'}(N)$ is a definable group homomorphism between algebraic groups. However, it is not necessarily a morphism of algebraic groups. For example, $\Phi$ may be the graph of the inverse Frobenius map on $G = G' = \Ga$. 
The inverse of the Frobenius is essentially the only such example: composing with a high enough power of the Frobenius makes any definable group homomorphism as above into a morphism of algebraic varieties. See \Cref{lemma: Def hom} (or \cite[Section~4.5]{Poizat}).
In particular, definable group homomorphisms are Zariski-continuous.%
\end{remark}

\begin{definition}
Two quasi-isomorphisms $\Phi, \Psi \le G \times G'$ are \emph{equivalent}, denoted $\Phi \bumpeq \Psi$, if there exists a finite normal subgroup $N \triangleleft G \times G'$ such that
\[\Phi N = \Psi N.\]
\end{definition}

\begin{remark}
Note that if instead we have $\Phi M = \Psi N$ for finite normal $M,N \triangleleft G \times G'$ then, by replacing both $M,N$ with $M\cdot N$, we obtain that $\Phi\bumpeq \Psi$. We also have that for connected groups $G,G'$, $\Phi \bumpeq \Psi$ if and only if $\Phi^0 = \Psi^0$ as subgroups of $G \times G'$, and $\bumpeq$ is an equivalence relation.
\end{remark}

Quasi-isomorphisms $\Phi \le G \times G'$ and $\Psi \le G' \times G''$ can be composed:
\[
  \Psi \circ \Phi = \Set{ (x,z) \in G \times G'' \mid \exists y \in G': (x,y) \in \Phi \text{ and } (y,z) \in \Psi }.
\]
It is routine to check that this is a quasi-isomorphism $G \to G''$. The inverse of $\Phi$ is the opposite relation $\Phi^{-1} = \Set{ (y,x) \in G' \times G \mid (x,y) \in \Phi }$. These operations turn the set $\QAut(G)$ of equivalence classes of quasi-automorphisms of $G$ into a group whose identity is the equivalence class of the diagonal $\Delta = \Set{ (x,x) \in G \times G }$.

\begin{notation}
    By an abuse of notation, we sometimes treat quasi-isomorphisms as functions in the following way. Let $\Phi \le G \times G'$ be a quasi-isomorphism. By \Cref{rem:quasi_iso_as_map}, we obtain a well-defined homomorphism $G \to G'/\pi_{G'}(N)$, where $N = \ker[\pi_G\restriction_\Phi\colon \Phi \to G]$. Temporarily denote this homomorphism $f_\Phi$. For $x \in G$ and $y\in G'$ we denote the statement $f_\Phi(x) = y \cdot \pi_{G'}(N)$~by
    \[
      \Phi(x)\bumpeq y.
    \]
    Let $\Psi \le G' \times G''$ be another quasi-isomorphism. From $\Psi$ we get an additional homomorphism $f_\Psi\colon G' \to G'' / \pi_{G''}(M')$, where $M' = \ker[\pi_{G'} \restriction_\Psi \colon \Psi \to G']$. We then have a well-defined homomorphism
    \[f_{\Psi \circ \Phi}\colon G \to G''/\left(f_\Psi(\pi_{G'}(N))\cdot \pi_{G''}(M')\right).\]
    Using this, we apply function notation to $\Psi\circ \Phi$, i.e., we use $(\Psi \circ \Phi)(x) \bumpeq y$ to denote the situation where $f_{\Psi \circ \Phi}(x) = y \cdot f_\Psi(\pi_{G'}(N))\cdot \pi_{G''}(M')$.

    This generalizes to various other situations. For instance, we use notation such as $\Phi(x) \bumpeq \Psi(x')$. The implicit normal subgroups are always finite and independent of the specific arguments of the ``functions'' appearing in the equation. Their identity is not important so much as their existence, which can be easily seen on a case by case basis.
\end{notation}

Let $G$ be a one-dimensional, connected algebraic group over~$F$. Recall that its $F$-definable endomorphisms form a ring $\End_F(G)$ which satisfies the left Ore condition: for every $\phi,\psi \in \End_F(G)$ there exist $\alpha,\beta \in \End_F(G)$ such that $\alpha\circ \phi = \beta\circ \psi$. Further, if $\phi$ is nonzero then $\beta$ is nonzero, and if $\psi$ is nonzero then $\alpha$ is nonzero. This property guarantees that $\End_F(G)$ embeds into a canonical skew field of fractions $L_F(G)$ whose elements are the usual equivalence classes of left-fractions $\psi^{-1}\phi$ for some nonzero~$\psi$; see e.g.~\cite[p.~119]{Jacobson_BA1} (although it uses the right-fraction convention) or \cite{Herstein_noncommutative_rings}.
The next result shows that the group of quasi-automorphisms is indeed compatible with the multiplicative group of the skew field in dimension~one.

\begin{proposition} \label{prop: quasi-aut as skew field}
Let $G$ be a one-dimensional, connected algebraic group, written additively. There is an isomorphism $L_F(G)^\times \isoto \QAut(G)$ which associates to a fraction $\psi^{-1}\phi \in L_F(G)^\times$ the equivalence class of the quasi-automorphism $\Phi = \Set{ (x,y) \in G \times G \mid \phi(x) = \psi(y) }$.
\end{proposition}

\begin{proof}
Define $\Phi(\psi, \phi) = \Set{ (x,y) \in G \times G \mid \phi(x) = \psi(y) }$ for nonzero $\psi, \phi \in \End_F(G)$ (note the ordering of $\phi,\psi$ in the definition of $\Phi(\psi,\phi)$). Each such set is clearly an $F$-definable subgroup of~$G \times G$. Since $G$ is one-dimensional and connected, a nonzero endomorphism has one-dimensional image and is therefore surjective with finite kernel. This implies that $\Phi(\psi, \phi)$ is a quasi-automorphism. We must first show that this map induces a well-defined map between equivalence classes.

\begin{claim}
If $\psi_1^{-1}\phi_1 = \psi_2^{-1}\phi_2$ in $L_F(G)^\times$ then $\Phi(\psi_1, \phi_1) \bumpeq \Phi(\psi_2, \phi_2)$.
\end{claim}

\begin{proof}
By the Ore condition there exist $\alpha, \beta \in \End_F(G)$, both nonzero, such that $\alpha \psi_1 = \beta \psi_2$. Now from the equation $(\alpha \psi_1)^{-1} (\alpha \phi_1) = (\beta \psi_2)^{-1}(\beta \phi_2)$ in $L_F(G)$, on multiplying from the left by $\alpha \psi_1 = \beta \psi_2$, we find $\alpha \phi_1 = \beta \phi_2$ (in $L_F(G)$ and hence in~$\End_F(G)$ as it embeds in~$L_F(G)$). Observe that, since $\alpha \neq 0$,
\[
  \leftindex_\alpha \Phi(\psi_1, \phi_1) \coloneqq \Phi(\alpha \psi_1, \alpha \phi_1) \bumpeq \Phi(\psi_1, \phi_1).
\]
Indeed, the inclusion $\supseteq$ holds, and conversely $\leftindex_\alpha \Phi(\psi_1, \phi_1) \subseteq \Phi(\psi_1, \phi_1) \cdot (\phi_1^{-1}(\ker \alpha) \times \psi_1^{-1}(\ker \alpha))$. Thus, $\Phi(\psi_1, \phi_1)$ and $\leftindex_\alpha \Phi(\psi_1, \phi_1)$ coincide up to a finite normal subgroup. Similarly, $\leftindex_\beta \Phi(\psi_2, \phi_2) \bumpeq \Phi(\psi_2, \phi_2)$. But then
\begin{align*}
  \Phi(\psi_2, \phi_2) \bumpeq \leftindex_\beta \Phi(\psi_2, \phi_2) &= \Set{ (x,y) \in G \times G \mid \beta \phi_2(x) = \beta \psi_2(y) } \\
  &= \Set{ (x,y) \in G \times G \mid \alpha \phi_1(x) = \alpha \psi_1(y) } = \leftindex_\alpha \Phi(\psi_1, \phi_1) \bumpeq \Phi(\psi_1, \phi_1).
\end{align*}
Since $\bumpeq$ is an equivalence relation, it follows that the map $\psi^{-1}\phi \mapsto [\Phi(\psi, \phi)]_\bumpeq$ is well-defined.
\end{proof}

Next we focus on the homomorphism property. Let $\psi_1^{-1}\phi_1, \psi_2^{-1}\phi_2 \in L_F(G)^\times$ be given. Using the Ore condition, choose nonzero $\alpha, \beta \in \End_F(G)$ such that $\alpha \phi_1 = \beta \psi_2$. This gives a way to express the product $(\psi_1^{-1} \phi_1) (\psi_2^{-1} \phi_2) = (\alpha \psi_1)^{-1} (\beta \phi_2)$. It suffices to prove the following.

\begin{claim}
The equivalence $\Phi(\alpha \psi_1, \beta \phi_2) \bumpeq \Phi(\psi_1, \phi_1) \circ \Phi(\psi_2, \phi_2)$ holds.
\end{claim}

\begin{proof}
Using the group structure of $\QAut(G)$ and the fact $\leftindex_\alpha \Phi(\mathord-) \bumpeq \Phi(\mathord-)$ derived earlier, it~follows that
\begin{align*}
  &\hphantom{{}={}} \Phi(\psi_1, \phi_1) \circ \Phi(\psi_2, \phi_2) \bumpeq \leftindex_\alpha \Phi(\psi_1, \phi_1) \circ \leftindex_\beta \Phi(\psi_2, \phi_2) \\
  &= \{(x,z) \in G \times G \mid \exists y\in G: (x,y) \in \leftindex_\beta \Phi(\psi_2,\phi_2) \land (y,z) \in \leftindex_\alpha \Phi(\psi_1, \phi_1)\} \\
  &= \{(x,z) \in G \times G \mid \exists y \in G: \beta\phi_2(x) = \beta\psi_2(y) \land \alpha\phi_1(y) = \alpha \psi_1(z)\} \\
  &= \Set{ (x,z) \in G \times G \mid \beta\phi_2(x) = \alpha\psi_1(z) } = \Phi(\alpha\psi_1, \beta\phi_2),
\end{align*}
where to get to the last line from the previous one we have used the fact that $\beta\psi_2(y)=\alpha\phi_1(y)$ for all $y \in G$.
\end{proof}

It remains to show that the map $L_F(G)^\times \to \QAut(G)$ is bijective. Given the homomorphism property, injectivity reduces to the following claim.

\begin{claim}
If $\Phi(\psi, \phi) \bumpeq \Delta$ then $\psi = \phi$.
\end{claim}

\begin{proof}
For an endomorphism $\phi \in \End_F(G)$ the quasi-automorphism $\Phi(1, \phi)$ is just its graph and hence a closed irreducible subgroup of~$G \times G$. Thus if $\psi, \phi \in \End_F(G)$ are distinct, their graphs $\Psi = \Phi(1, \psi)$ and $\Phi = \Phi(1, \phi)$ are also distinct.
If $N \le G\times G$ is a finite subgroup, $N=\{n_1,\ldots,n_t\}$, then $\Psi N = \bigcup_{n \in N} \Psi n$ is a finite union of cosets. If $\Psi N = \Phi N$, we have $\Phi \subseteq \Psi n$ for some $n \in N$, as $\Phi \subseteq \Phi N = \Psi N$ and $\Phi$ is irreducible; but then $(e_G, e_G) \in \Psi n$, so $n \in \Psi$, and $\Phi \subseteq \Psi$. In the same way we see that $\Psi \subseteq \Phi$. Hence, $\Phi(1, \psi) \bumpeq \Phi(1, \phi)$ implies $\Phi(1, \psi) = \Phi(1, \phi)$. Thus, if $\psi \neq \phi$ then $\Phi(\psi, \phi) = \Phi(1, \psi)^{-1} \circ \Phi(1, \phi) \centernot\bumpeq \Delta$ as required.
\end{proof}

Finally we have to prove surjectivity. Let $\Phi \le G \times G$ be any quasi-automorphism. Since $G$ is commutative, so is $G \times G$, and every subgroup in it is normal. Let $H = (G \times G) / \Phi$ be the quotient group (with addition as the operation) and $\pi\colon G\times G \to H$ the associated quotient map. The group $H$ has dimension $\dim(G \times G) - \dim \Phi = 1$ and is connected as a continuous image of a connected group. Consider the morphism $f_1\colon G \to H$ given by $f_1(x) = \pi(x,0)$. It is dominant because otherwise it would be trivial and $\Phi = \ker \pi \supseteq G \times \{e_G\}$ which contradicts the quasi-automorphism properties. Thus, $f_1$ is an isogeny $G \to H$. Analogously, $f_2\colon G \to H$, $f_2(x) = -\pi(0,x)$ is an isogeny.
We have $\pi(x,y) = f_1(x) - f_2(y)$ and thus
\[
  \Phi = \ker \pi = \Set{ (x,y) \in G \times G \mid f_1(x) = f_2(y) }
\]
but $f_1, f_2$ are only isogenies $G \to H$ and not endogenies of~$G$.

\begin{claim}
There is an $F$-definable isogeny $\alpha\colon H \to G$.
\end{claim}

\begin{proof}
Fix any $F$-definable isogeny $G \to H$, say $f_1$, and let $N = \ker f_1 \triangleleft G$ be its finite kernel. We use the fact that there exists a nonzero $\beta \in \End_F(G)$ with $\ker \beta \supseteq N$. This relies on details about the connected, one-dimensional algebraic groups over~$F$:
\begin{itemize}[itemsep=0.4em]
\item If $G \simeq \Gm$ or an elliptic curve, then $\End_F(G)$ has characteristic zero. It suffices to pick the nonzero endomorphism $\beta = \card{N} \in \ZZ \subseteq \End_F(G)$. Since every element of $N$ has order dividing $\card{N}$, the kernel surely contains $N$. %
\item For $G \simeq \Ga$ recall that $\End_F(G) \supseteq \End_F^\mathrm{alg}(G) \simeq F[x; \Frob]$. We construct $\beta$ iteratively, starting with $N_1 = N$. As long as there exists $0 \neq n \in N_i$, pick some element $\beta_i = x - a_i \in \End_F(G)$ which annihilates~$n$. This is possible because nonzero $a \in F$ are units in $\End_F(G)$ and therefore bijective on~$G$. Thus, $N_{i+1} = \beta_i(N_i)$ has strictly lower cardinality than~$N_i$. Eventually, $N_{d+1} = \{0\}$ and $\beta = \beta_d \cdots \beta_1$ is the required endogeny.
\end{itemize}
Thus we obtain a quotient map $q\colon H \simeq G/N \to G/\ker\beta$ and an induced map $\overline{\beta}$ from the first isomorphism theorem for algebraic groups:
\begin{center}
\begin{tikzcd}
G \arrow[rr, "\beta"] \arrow[d]               &  & G \\
H \arrow[d, "q"']                             &  &   \\
G/\ker\beta \arrow[rruu, "\overline{\beta}"'] &  &  
\end{tikzcd}
\end{center}
Everything is $F$-definable and $\alpha = \overline{\beta} \circ q$ is the desired~isogeny.
\end{proof}

Now form the endogenies $\phi = \alpha \circ f_1$ and $\psi = \alpha \circ f_2$ of $G$ and observe that
\begin{align*}
  \Phi &= \Set{ (x,y) \in G \times G \mid f_1(x) = f_2(y) } \\
  &\bumpeq \Set{ (x,y) \in G \times G \mid (\alpha \circ f_1)(x) = (\alpha \circ f_2)(y) } = \Phi(\psi, \phi).
\end{align*}
This shows surjectivity and completes the proof.
\end{proof}

\subsection{The quasi-automorphism labeling} %
\label[section]{sec: quasi-aut labels}
Consider an $F$-definable group $G$ over $\mathbb{K} \supseteq F$. If $G$ is one-dimensional, any generic pair of elements $x,y \in G$ lies in an Evans--Hrushovski projective line coordinatized by $L_F(G)$. We can coordinatize $x$ by $[1:0]$ and $y$ by $[0:1]$, so that every other point $z$ in this projective line has some coordinates $[-\varphi : \psi]=[-\psi^{-1}\varphi : 1]$, where $\psi^{-1}\varphi \in L_F(G)$. Our goal is to produce similar coordinates when $G$ is of arbitrary dimension. More specifically, given a group configuration with an extra point, we show how to add points and lines in a way that implies this point can be labelled by a quasi-automorphism $\Phi \in \QAut(G)$, where $\Phi$ satisfies $z \bumpeq \Phi(x)^{-1}y$. 
\begin{remark}
\label[remark]{rem: quasi-aut coordinates}
    In the special case where $G$ is one-dimensional and 
\[\Phi = \{(x,y) \in G \times G \mid \varphi(x) = \psi(y)\},\]
we obtain $z \bumpeq (\psi^{-1} \varphi)(x)^{-1} y$, i.e. $\psi(z) = \varphi(x)^{-1}\psi(y)$, and hence $z$ is interalgebraic with $\varphi(x)^{-1}\psi(y)$. So $z$ has homogeneous coordinates $[-\varphi:\psi]$.
The minus sign in $[-\varphi:\psi]$ seems unavoidable: we can't compensate for it by ``negating'', because the inversion map is not an endomorphism in a noncommutative group.
\end{remark}

\begin{definition}
Let $F\subset\mathbb{K}$ be fields. Two group triples $(x,y,z)$,
$(x',y',z')$ in $\mathbb{K}$ over $F$ are \emph{interdefinable} if they
are elementwise interdefinable, meaning
\[
  \dcl(Fx) = \dcl(Fx'), \quad
  \dcl(Fy) = \dcl(Fy'), \quad \text{and }
  \dcl(Fz) = \dcl(Fz').
\]
\emph{Interalgebraicity} of group triples is also understood elementwise.
\end{definition}

\begin{theorem} \label{thm: quasi-iso}
Let $(x,y,z)$ and $(x',y',z')$ be two group triples in $\mathbb{K}$, with groups $G$ and $G'$ respectively. Suppose that $x' \in \acl_F(x)$, $y' \in \acl_F(y)$ and $z' \in \acl_F(z)$. Then there are elements $a_0', b_0' \in G'$, definable over $F$, and a definable quasi-epimorphism $\Phi \le G \times G'$ such that $x \mathrel{\Phi} a_0' x'$, $y \mathrel{\Phi} y' b_0'$, and $z \mathrel{\Phi} a_0' z' b_0'$. In particular, the translated group triple $(a_0' x', y' b_0', a_0' z' b_0')$ is interdefinable with $(x', y', z')$.
Moreover, if $(x, y, z)$ and $(x', y', z')$ are interalgebraic then $\Phi$ is a quasi-isomorphism.
\end{theorem}

\begin{lemma}\label{lemma: types}
Let $a,b, x, y$ be tuples in $\mathbb{K}$ such that $\RM(ab/F) = \RM(a/F)+\RM(b/F)$ and similarly $\RM(xy/F) = \RM(x/F)+\RM(y/F)$. Assume $\tp(a/F)=\tp(x/F)$ and $\tp(b/F)=\tp(y/F)$. Then $\tp(ab/F)=\tp(xy/F)$.
\end{lemma}

\begin{proof}
The assumptions provide $\sigma \in \Aut(\mathbb{K}/F)$ with $\sigma(b) = y$. Moreover, $x \indep_F y$ and $a \indep_F b$ hold. Since $\ACF_p$ is stable, types over models are stationary. Hence $\tp(x/F)$ has a unique non-forking extension to $F(y)$ and by the independence this is just $\tp(x/Fy)$; the same holds for~$a$ and~$b$. The type $\sigma(\tp(a/Fb))$ is a non-forking extension of $\tp(a/F) = \tp(x/F)$ over $\sigma(F(b)) = F(y)$. Stationarity forces it to coincide with $\tp(x/Fy)$. This implies $\tp(xy/F) = \tp(ab/F)$.
\end{proof}

\begin{proof}[Proof of \Cref{thm: quasi-iso}]
By \Cref{rem: algebraic group} we may assume that $G$ and $G'$ are algebraic groups over~$F$ and endow them with the Zariski topology.
Fix $F$-definable charts $U$ on $G$ and $U'$ on $G'$ that contain $x$ and $x'$ respectively. Let $I \triangleleft F[U]\otimes_F F[U']$ be the ideal of polynomials that vanish on $x,x'$. Since $x' \in \acl_F(x)$ it is prime and nonzero; and its zero set $Z=Z(I)$ has $\dim(Z(I)) = \RM(Z(I)) = \RM(x,x') =\RM(x)$.
The zero set $Z \subsetneq U \times U'$ is a locally closed subset of $G \times G'$, whose closure $C \subset G \times G'$ satisfies $C \cap (U\times U') = Z$. Define a relation $a \Rel1 a' \defiff (a,a') \in C$; %
analogously $y,y'$ and $z,z'$ define relations $\Rel2$ and $\Rel3$, respectively. Each of these relations is a Zariski-closed subset of $G \times G'$. Consider the closed subset $S = \mathord{\Rel1} \times \mathord{\Rel2}$ of the product $(G \times G')^2$. If $(a, a', b, b') \in S$ is generic then $a b \Rel3 a' b'$ by \Cref{lemma: types}. Since $\mathord{\Rel3}$ is closed and the group operations are continuous, this property holds for the entire set~$S$. Analogously there are continuous maps $\mathord{\Rel1} \times \mathord{\Rel3} \to \mathord{\Rel2}$ and $\mathord{\Rel2} \times \mathord{\Rel3} \to \mathord{\Rel1}$. They yield the following implications for arbitrary tuples from $(G \times G')^2$:
\begin{align}
  \label{eq: I+II=>III}
  a \Rel1 a' \land b \Rel2 b' &\implies a b \Rel3 a' b', \\
  \label{eq: I+III=>II}
  a \Rel1 a' \land c \Rel3 c' &\implies a^{-1} c \Rel2 {a'}^{-1} c', \\
  \label{eq: II+III=>I}
  b \Rel2 b' \land c \Rel3 c' &\implies c b^{-1} \Rel1 c' {b'}^{-1}.
\end{align}
The implications \eqref{eq: I+II=>III}--\eqref{eq: II+III=>I} are referred to as the ``two-of-three property'' for the relations $\oRel1, \oRel2, \oRel3$ for the remainder of this proof.

For any tuple $(e_G, a_0', e_G, b_0') \in S$ define modified relations
\begin{align*}
  a \Rel1' a' &\defiff a \Rel1 a_0' a', \\
  b \Rel2' b' &\defiff b \Rel2 b' b_0', \\
  c \Rel3' c' &\defiff c \Rel3 a_0' c' b_0'.
\end{align*}
These new relations satisfy $e_G \Rel1' e_{G'}$, $e_G \Rel2' e_{G'}$ and $e_G \Rel3' e_{G'}$. It is easy to check that they also satisfy the two-of-three property. This forces all three relations to coincide. Indeed, if $a \Rel1' a'$ then $e_G \Rel2' e_G{'}$ and \eqref{eq: I+II=>III} yield $a = a e_G \Rel3' a' e_{G'} = a'$, so $\mathord{\Rel1'} \subseteq \mathord{\Rel3'}$. By symmetric arguments we get $\mathord{\Rel1'} = \mathord{\Rel2'} = \mathord{\Rel3'}$.

The set $\Phi = \mathord{\Rel1'}$ contains the identity $(e_G, e_{G'})$ of $G \times G'$ and is closed under multiplication and inverses, using the two-of-three property together with the fact that $\mathord{\Rel1'} = \mathord{\Rel2'} = \mathord{\Rel3'}$. Hence, it is a subgroup.
The~components of the element $(x, {a_0'}^{-1} x') \in \Phi$ are generic in $G$ and $G'$, respectively. Therefore, the projection maps $\pi\colon \Phi \to G$ and $\pi'\colon \Phi \to G'$ are dominant and indeed surjective since the image of a homomorphism of algebraic groups is closed. Since $\dim \Phi = \dim G$ it follows that $\pi$ has finite kernel which completes the proof that $\Phi$ is a quasi-epimorphism $G \to G'$. If~the group triples are even interalgebraic then $\dim \Phi = \dim G = \dim G'$ and it follows that $\Phi$ is a quasi-isomorphism.

Since $e_G$ and the relations $\Rel1$, $\Rel2$ are definable over $F$, the group elements $a_0', b_0'$ are algebraic over $F$; but $F$ is algebraically closed, so $a_0', b_0'$ are definable over~$F$. We have that $x \mathrel{\Phi} {a_0'}^{-1} x'$, $y \mathrel{\Phi} y' {b_0'}^{-1}$ and $z \mathrel{\Phi} {a_0'}^{-1} z' {b_0'}^{-1}$. This yields the statement of the theorem (replacing $a_0'$ and $b_0'$ by their inverses).
\end{proof}

\begin{corollary}
    \label{cor: quasi-iso config}
    Let $(a',b',c',x',y',z')$ be a group configuration for an $F$-definable group $G'$, and assume $(y',b',z')$ is elementwise interalgebraic with a group triple $(y,b,z)$ for an $F$-definable group $G$. Then there exist $a,c,x \in G$, elementwise interalgebraic with the corresponding primed elements, such that $(a,b,c,x,y,z)$ is a group configuration for $G$.
\end{corollary}
\begin{proof}
    By \Cref{thm: quasi-iso} there is an $F$-definable quasi-isomorphism $\Phi \le G\times G'$ and elements $a_0', b_0' \in G'$ definable over $F$, such that $y \mathrel\Phi a_0' y'$, $b \mathrel\Phi b' b_0'$, and $z \mathrel\Phi a_0' z' b_0'$. Since $F$ is algebraically closed, there exists $x \in G$ such that $x \mathrel\Phi a_0' x'$. Define
    \[a = x^{-1}y\quad \text{and}\quad  c = z^{-1} x.\]
    Then $(a,b,c,x,y,z)$ is a group configuration for $G$, and since $\Phi$ is a quasi-isomorphism, we have
    \[a\mathrel\Phi (a_0' x')^{-1} (a_0' y') = x'^{-1}y' = a' \quad \text{and}\quad c \mathrel \Phi (a_0' z' b_0')^{-1}(a_0' x') = b_0'^{-1} z'^{-1}x' = b_0'^{-1} c'.\]
    In particular, each unprimed element is interalgebraic with the corresponding primed one.
\end{proof}

Consider three superimposed group configurations as in \Cref{fig: triple stacked}. This situation can be encoded more succinctly in the diagram in~\Cref{fig: triple compact}. Note that $b_2 \in \acl(y_1, z_1)$; up to interalgebraicity this element appears in \Cref{fig: triple compact} as $b'$. Our goal is to relate it to the group structure of $G_1$ in a similar way to how $b_1 = y_1^{-1} z_1$.

\begin{figure}
\begin{tikzpicture}[3d view={10}{45}, every node/.style={outer sep=0pt}, elt/.style={fill, circle, thick, inner sep=3pt, outer sep=0pt}, every path/.style={thick}]]
  \node [elt, label={right:$x_1$}] (x1) at (0,2,0) {};
  \node [elt, label={left:$y_1$}] (y1) at (-3,-2,0) {};
  \node [elt, label={right:$z_1$}] (z1) at (3,-2,0) {};
  \node [elt, label={[yshift=0.5ex]left:$a_1$}] (a1) at (-1.5,0,0) {};
  \node [elt, label={below:$b_1$}] (b1) at (0,-2,0) {};
  \node [elt, label={right:$c_1$}] (c1) at (1.5,0,0) {};
  \node (G1) at (3.5,0.5,0) {$G_1$};

  \draw[ultra thick] (x1) -- (a1) -- (y1) -- (b1) -- (z1) -- (c1) -- (x1);
  \draw[ultra thick, cyan!80!black] (a1) to[out=280, in=170] (b1) to[out=0, in=240] (c1);

  \node [elt, label={right:$x_2$}] (x2) at (0,2,-5) {};
  \node [elt, label={left:$y_2$}] (y2) at (-3,-2,-5) {};
  \node [elt, label={right:$z_2$}] (z2) at (3,-2,-5) {};
  \node [elt, label={[yshift=0.5ex]left:$a_2$}] (a2) at (-1.5,0,-5) {};
  \node [elt, label={above:$b_2$}] (b2) at (0,-2,-5) {};
  \node [elt, label={right:$c_2$}]  (c2) at (1.5,0,-5) {};
  \node (G2) at (3.5,0.5,-5) {$G_2$};
  
  \draw[ultra thick] (x2) -- (a2) -- (y2) -- (b2) -- (z2) -- (c2) -- (x2);
  \draw[ultra thick, orange!90!black] (a2) to[out=280, in=170] (b2) to[out=0, in=240] (c2);

  \draw[dashed] (x1) -- (x2);
  \draw[dashed] (y1) -- (y2);
  \draw[dashed] (z1) -- (z2);
  \draw[dashed, blue] (c1) to node[draw, solid, midway, right, yshift=-0.5cm, outer sep=4pt] {$\Phi_{12}$} (c2);

  \node [elt, label={right:$x_3$}] (x3) at (0,2,-10) {};
  \node [elt, label={left:$y_3$}] (y3) at (-3,-2,-10) {};
  \node [elt, label={right:$z_3$}] (z3) at (3,-2,-10) {};
  \node [elt, label={[yshift=0.5ex]left:$a_3$}] (a3) at (-1.5,0,-10) {};
  \node [elt, label={below:$b_3$}] (b3) at (0,-2,-10) {};
  \node [elt, label={right:$c_3$}] (c3) at (1.5,0,-10) {};
  \node (G3) at (3.5,0.5,-10) {$G_3$};

  \draw[ultra thick] (x3) -- (a3) -- (y3) -- (b3) -- (z3) -- (c3) -- (x3);
  \draw[ultra thick, green!60!black] (a3) to[out=280, in=170] (b3) to[out=0, in=240] (c3);

  \draw[dashed] (x2) -- (x3);
  \draw[dashed] (y2) -- (y3);
  \draw[dashed] (z2) -- (z3);
  \draw[dashed, blue] (b2) to node[draw, solid, midway, right, yshift=-0.2cm, outer sep=3pt] {$\Phi_{23}$} (b3);

  \draw[dashed, blue] (a1) to[bend right=20] node[draw, solid, midway, left, outer sep=4pt] {$\Phi_{13}$} (a3);
\end{tikzpicture}
\caption{Three superimposed group configurations. The dashed lines indicate interalgebraicities. All three configurations share (up to interalgebraicity) the vertices of the triangle. In addition any two configurations share another point which establishes a quasi-isomorphism between the associated groups by \Cref{thm: quasi-iso}.}
\label{fig: triple stacked}
\end{figure}
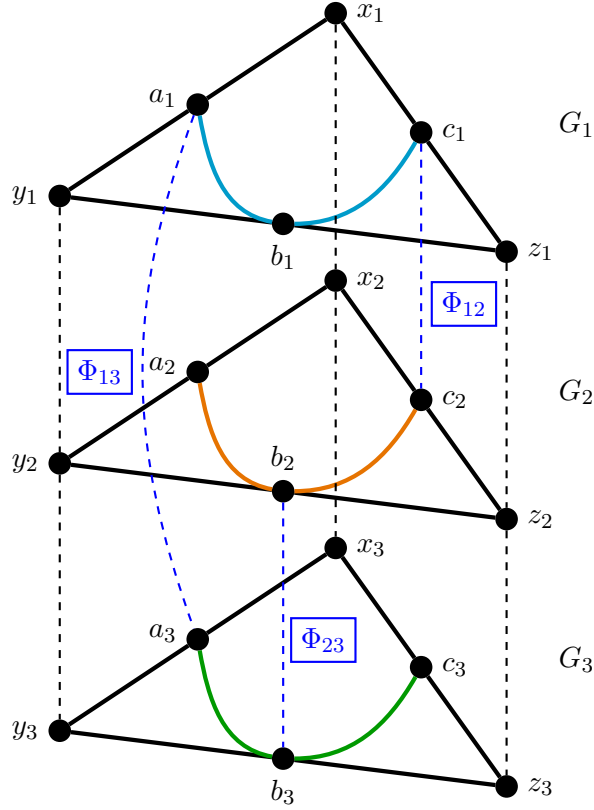

\begin{figure}
\begin{tikzpicture}[xscale=1.4, yscale=1.6, elt/.style={fill, circle, thick, inner sep=3pt, outer sep=0pt}, every path/.style={thick}]
  \node [elt, label={above:$x$}] (x) at (0,2) {};
  \node [elt, label={left:$y$}] (y) at (-3,-2) {};
  \node [elt, label={right:$z$}] (z) at (3,-2) {};
  \node [elt, label={left:$a$}] (a) at (-2,-0.66) {};
  \node [elt, label={left:$a'$}] (a') at (-1,0.66) {};
  \node [elt, label={below:$b$}] (b) at (1,-2) {};
  \node [elt, label={below:$b'$}] (b') at (-1,-2) {};
  \node [elt, label={right:$c$}] (c) at (2,-0.66) {};
  \node [elt, label={right:$c'$}] (c') at (1,0.66) {};

  \draw[ultra thick] (x) -- (a') -- (a) -- (y) -- (b') -- (b) -- (z) -- (c) -- (c') -- (x);
  \draw[ultra thick, cyan!80!black] (a) to[out=300, in=180] (b) to[out=0, in=270] (c);
  \draw[ultra thick, green!60!black] (b') to[out=170, in=270] (a) to[out=50, in=180] (c');
  \draw[ultra thick, orange!90!black] (a') to[out=10, in=130] (c) to[out=285, in=0] (b');

  \node (Phi12) [draw, blue, right=of c, xshift=-2ex] {$\Phi_{12}$};
  \node (Phi23) [draw, blue, below=of b', yshift=2ex] {$\Phi_{23}$};
  \node (Phi13) [draw, blue, left=of a, xshift=2ex] {$\Phi_{13}$};
\end{tikzpicture}
\caption{The same configuration as in \Cref{fig: triple stacked} with interalgebraic elements suppressed. For instance $x_1, x_2, x_3$ from the other figure occupy the spot of $x$ here, $a_1$ and $a_3$ occupy $a$, and $a_2$ is $a'$ and so on. Each colored arc belongs to a distinct group. The three points which are contained in two of these arcs deliver quasi-isomorphisms.}
\label{fig: triple compact}
\end{figure}
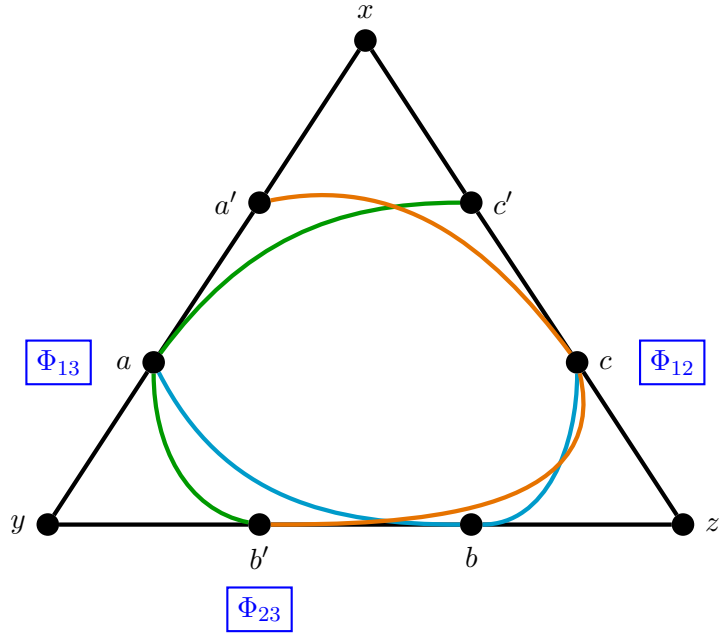

\begin{theorem} \label{thm: quasi-aut}
In the situation of \Cref{fig: triple stacked,fig: triple compact} there exist a quasi-automorphism $\Phi \le G_1 \times G_1$, a quasi-isomorphism $\Phi_{21} \le G_2 \times G_1$ and an $F$-definable $t \in G_1$ such that $\Phi_{21}(b_2) \cdot t \bumpeq \Phi(y_1)^{-1} z_1$. In particular $b_2$ and $\Phi(y_1)^{-1} z_1$ are interalgebraic.
\end{theorem}

\begin{proof}
The proof consists of three applications of \Cref{thm: quasi-iso} to interalgebraic group triples which complete a round trip through the three layers of \Cref{fig: triple stacked}.
\begin{enumerate}
\item
Apply it to the interalgebraic group triples $(y_1, a_1^{-1}, x_1)$ and $(y_3, a_3^{-1}, x_3)$. This gives $F$-definable elements $s_3, t_3 \in G_3$ and a quasi-isomorphism $\Phi_{13} \le G_1 \times G_3$ such that
\begin{alignat*}{6}
  y_1 &{}\mathrel{\Phi_{13}} \hat y_3 \defas s_3 y_3, &\quad& a_1^{-1} &{}\mathrel{\Phi_{13}} \hat a_3^{-1} \defas a_3^{-1} t_3, &\quad& x_1 &{}\mathrel{\Phi_{13}} \hat x_3 \defas s_3 x_3 t_3.
\end{alignat*}
In addition define $\hat z_3 \defas s_3 z_3$ so that $\hat y_3^{-1} \hat z_3 = y_3^{-1} z_3 = b_3$. Thus $(\hat y_3, b_3, \hat z_3)$ is a group triple for $G_3$ which is elementwise interdefinable with $(y_3, b_3, z_3)$.

\item
Similarly obtain a quasi-isomorphism $\Phi_{32} \le G_3 \times G_2$ from the triples $(\hat y_3, b_3, \hat z_3)$ and $(y_2, b_2, z_2)$. Thus there are $F$-definable $s_2, t_2 \in G_2$ such that
\begin{alignat*}{6}
  \hat y_3 &{}\mathrel{\Phi_{32}} \hat y_2 \defas s_2 y_2, &\quad& b_3 &{}\mathrel{\Phi_{32}} \hat b_2 \defas b_2 t_2, &\quad& \hat z_3 &{}\mathrel{\Phi_{32}} \hat z_2 \defas s_2 z_2 t_2.
\end{alignat*}
With $\hat x_2 \defas s_2 x_2$ it follows that $\hat z_2^{-1} \hat x_2 = t_2^{-1} c_2 \asdef \hat c_2$.

\item
\Cref{thm: quasi-iso} also applies to the configurations $(\hat x_2^{-1}, \hat z_2, \hat c_2^{-1})$ and $(x_1^{-1}, z_1, c_1^{-1})$. We obtain $F$-definable elements $s_1, t_1 \in G_1$ and a quasi-isomorphism $\Phi_{21} \le G_2 \times G_1$ such that
\begin{alignat*}{6}
  \hat x_2^{-1} &{}\mathrel{\Phi_{21}} \hat x_1^{-1} \defas s_1 x_1^{-1}, &\quad& \hat z_2 &{}\mathrel{\Phi_{21}} \hat z_1 \defas z_1 t_1, &\quad& \hat c_2^{-1} &{}\mathrel{\Phi_{21}} \hat c_1^{-1} \defas s_1 c_1^{-1} t_1.
\end{alignat*}
\end{enumerate}

Modulo the finite kernels of these quasi-isomorphisms we have $\hat y_2 \bumpeq (\Phi_{32} \circ \Phi_{13})(y_1)$ and $\hat z_1 \bumpeq \Phi_{21}(\hat z_2)$. Thus with the quasi-automorphism $\Phi = \Phi_{21} \circ \Phi_{32} \circ \Phi_{13}$ of $G_1$ this yields
\begin{align*}
  \Phi(y_1)^{-1} \cdot \hat z_1 &\bumpeq \Phi_{21}(\hat y_2^{-1} \hat z_2) \bumpeq \Phi_{21}(y_2^{-1} s_2^{-1} s_2 z_2 t_2) \\
  &\bumpeq \Phi_{21}(b_2 t_2) \bumpeq \Phi_{21}(b_2) \cdot \Phi_{21}(t_2).
\end{align*}
Hence $\Phi(y_1)^{-1} z_1 \bumpeq \Phi_{21}(b_2) \cdot [\Phi_{21}(t_2) \cdot t_1^{-1}]$. The bracketed term on the right-hand side is definable over~$F$.
\end{proof}

\begin{remark}
The quasi-automorphism in \Cref{thm: quasi-aut} is unique up to $\bumpeq$-equivalence given the $\bumpeq$-class of $\Phi_{21}(b_2)$. To see this, suppose that $\Phi, \Phi' \le G_1 \times G_1$ are two quasi-automorphisms such that for $F$-definable $t, t' \in G_1$ we have $\Phi(y_1)^{-1} z_1 \bumpeq \Phi_{21}(b_2) \cdot t$ and $\Phi'(y_1)^{-1} z_1 \bumpeq \Phi_{21}(b_2) \cdot t'$. Then $\Phi(y_1) \bumpeq t^{-1} \cdot t' \cdot \Phi'(y_1)$. After quotienting by a finite normal subgroup of $G_1$, this relation becomes an equality of maps in $y_1$ defined over~$F$. These maps are definable group homomorphisms and by \Cref{lemma: Def hom} Zariski-continuous. From the generic $y_1 \in G_1$ the equality extends by continuity to all of $G_1$ and for $e_{G_1}$ it implies $t \bumpeq t'$ which in turn ensures $\Phi \bumpeq \Phi'$.
\end{remark}

\begin{remark}
The gadget we construct contains the one constructed in \cite[Figure 12]{EH} and is compatible with it, although our construction is slightly less efficient, requiring an extra point.
\end{remark}

\subsection{Quasi-automorphisms induced through a quasi-epimorphism}
Consider two configurations of the type \Cref{fig: triple stacked} which are overlaid, where the primary configuration is named as in the figure and the secondary configuration is primed, for instance there is $y_1, y_2$ in the primary configuration and $y_1', y_2'$ in the secondary configuration. Assume each primed element is algebraic over the corresponding unprimed one. Thus we have groups $G_1$, $G_2$, and $G_3$, as well as $G_1'$, $G_2'$, and $G_3'$. \Cref{thm: quasi-aut} applies to the unprimed (``primary'') as well as to the primed (``secondary'') configuration separately: for the primary configuration, we get two quasi-isomorphisms $S = \Phi_{21}^{-1} \circ \Phi$ and $R = \Phi_{21}^{-1}$ from $G_1$ to $G_2$. For the secondary configuration we get $S' = \Phi_{21}'^{-1} \circ \Phi'$ and $R' = \Phi_{21}'^{-1}$ from $G_1'$ to $G_2'$. These satisfy:
\begin{gather}
  \label{eq: Gm transfer 1}
  S(y_1)^{-1} R(z_1) \bumpeq b_2 t, \;\text{ and }\; S'(y_1')^{-1} R'(z_1') \bumpeq b_2' t'.
\end{gather}
Moreover, \Cref{thm: quasi-iso}, applied once to $(y_1^{-1}, z_1, b_1)$ and once to a group triple in which $b_2$ is the middle element (e.g. $(y_2, b_2, z_2)$), provides also quasi-epimorphisms $\Psi_1 \le G_1 \times G_1'$ and $\Psi_2 \le G_2 \times G_2'$ with
\begin{gather}
  \label{eq: Gm transfer 2}
  \Psi_1(y_1)^{-1} \bumpeq s_1' y_1'^{-1}, \quad \Psi_1(z_1) \bumpeq z_1' t_1', \;\text{ and }\; \Psi_2(b_2) \bumpeq b_2' t_2',
\end{gather}
where the translations $t, t', s_1', \ldots$ are definable over~$F$.

\begin{theorem} \label{thm: Gm transfer}
In this situation $\Psi_1 \circ \Phi \bumpeq C \circ \Phi' \circ \Psi_1$ holds for an inner automorphism $C(x) = c^{-1} \cdot x \cdot c$ by an $F$-definable element $c \in G_1'$.
\end{theorem}

\begin{proof}
The given relations \eqref{eq: Gm transfer 1}--\eqref{eq: Gm transfer 2} imply
\begin{equation}
  \label{eq: Gm bumpeq}
  \Psi_2(S(y_1^{-1})) \Psi_2(R(z_1) t^{-1}) t_2'^{-1} \bumpeq b_2' \bumpeq S'(s_1'^{-1} \Psi_1(y_1^{-1})) R'(\Psi_1(z_1) t_1'^{-1}) t'^{-1}.
\end{equation}
For any finite normal subgroup $N' \triangleleft G_2'$ with its canonical projection $\pi\colon G_2' \to G_2'/N'$ the maps $f_1 = \pi \circ \Psi_2 \circ S$, $f_2 = \pi \circ S' \circ \Psi_1$, $g_1 = \pi \circ \Psi_2 \circ R$, and $g_2 = \pi \circ R' \circ \Psi_1$ are all homomorphisms of algebraic groups. We may choose $N' \triangleleft G_2'$ finite but big enough (depending on all the quasi-morphisms involved) so that \eqref{eq: Gm bumpeq} becomes
\[
  f_2(y_1^{-1})^{-1} \cdot \pi(c_1) \cdot f_1(y_1^{-1}) = g_2(z_1) \cdot \pi(c_2) \cdot g_1(z_1)^{-1} \text{ in $G_2'/N'$},
\]
with some constants $c_1, c_2 \in G_2'$ definable over~$F$.
Since $(y_1^{-1}, z_1) \in G_1 \times G_1$ is generic and both sides are Zariski-continuous, the equality $f_2(a)^{-1} \cdot \pi(c_1) \cdot f_1(a) = g_2(b) \cdot \pi(c_2) \cdot g_1(b)^{-1}$ holds for all $a, b \in G_1$. In particular the choices $a = e_{G_1}$ and $b = e_{G_1}$ show that $\pi(c_1) f_1 = f_2 \pi(c_2)$ and $\pi(c_1) g_1 = g_2 \pi(c_2)$ and even $\pi(c_1) = \pi(c_2)$, so $c_1 \bumpeq c_2$.
Unraveling the definitions of $f_1, f_2, g_1, g_2$ and writing $C_1(a) = c_1^{-1} \cdot a \cdot c_1$ for the inner automorphism this reads
\[
  \Psi_2 \circ S \bumpeq C_1 \circ S' \circ \Psi_1, \quad \Psi_2 \circ R \bumpeq C_1 \circ R' \circ \Psi_1.
\]

If $\Theta$ is any quasi-epimorphism then $C_1 \circ \Theta \bumpeq \Theta \circ C_2$ for some other inner automorphism~$C_2$ on the domain of~$\Theta$. This allows us to shuffle inner automorphism around as needed and deduce
\begin{align*}
  R'^{-1} \circ (\Psi_2 \circ S) &\bumpeq C_2 \circ R'^{-1} \circ S' \circ \Psi_1 \bumpeq C_2 \circ \Phi' \circ \Psi_1, \;\text{ and also} \\
  (R'^{-1} \circ \Psi_2) \circ S &\bumpeq C_3 \circ \Psi_1 \circ R^{-1} \circ S \bumpeq C_3 \circ \Psi_1 \circ \Phi,
\end{align*}
for two $F$-definable inner automorphisms $C_2$ and $C_3$. Note that in these manipulations the inner automorphisms are only ever pushed back and forth through quasi-isomorphisms. This finally gives
\[
  \Psi_1 \circ \Phi \bumpeq C \circ \Phi' \circ \Psi_1
\]
for $C = C_3^{-1} \circ C_2$ another $F$-definable inner automorphism.
\end{proof}

\section{The affine group configuration}
\label{sec: Aff}

\subsection{Facts about the affine group}

For a fixed integer $k \neq 0$ consider the algebraic group $G = \Gm \ltimes_\theta \Ga$ where $s \in \Gm$ acts on $\Ga$ via $s.a = \theta_s(a) = s^k a$. For $k = 1$ this is the \emph{affine group} $\Aff$, and in any case it is a connected extension of~$\Aff$ by a finite group. \Cref{sec: Aff recover} shows how such a group can be recovered from an augmented rank-2 group configuration. In this preliminary section we collect results about its group structure and quasi-automorphisms that we need later.

\begin{lemma} \label{lemma: ext Aff}
Let $G = \Gm \ltimes_\theta \Ga$ with $\theta_s(a) = s^k a$ for some $k \neq 0$.
\begin{enumerate}
\item \label{ext Aff: center} The center $Z(G)$ is the kernel of the homomorphism $\tau\colon G \to \Aff$ given by $\tau(s,a) = (s^k, a)$.
\item \label{ext Aff: commutators} The derived subgroup $[G,G]$ equals $N = \Set{ (1,a) \in G }$ and is isomorphic to~$\Ga$.
\item \label{ext Aff: normal}
Every normal subgroup of $G$ either contains~$N$ or is contained in $Z(G)$.
\item \label{ext Aff: complements} The complements of $N$ in $G$ are all conjugate via elements of~$[G,G]$.
\end{enumerate}
\end{lemma}

\begin{proof}
(\ref{ext Aff: center})
For $(s,a)$ to be in the center, the commutator $[(s,a), (t,b)] = (1, b(s^k - 1) - a(t^k - 1))$ must evaluate to~$(1,0)$ for all $(t,b) \in G$. For $t = b = 1$ we see that $s^k = 1$ is necessary. Over an algebraically closed field there exists $t \in \Gm$ which is not a $k$-th root of unity. This shows that also $a = 0$ is necessary. Hence $Z(G) \subseteq \ker \tau = \Set{ (s,0) \in G \mid s^k = 1 }$. The reverse inclusion is immediate.
Thus, the center is a subgroup of $\Gm$ and therefore cyclic. Its cardinality depends on $k$ and the characteristic~$p$: if $k = p^e k'$ with $p \nmid k'$ then $Z(G) \simeq \ZZ/k'\ZZ$.

(\ref{ext Aff: commutators})
The set $N$ is clearly a subgroup of~$G$ which contains every commutator, so $[G,G] \subseteq N$. For the other inclusion choose any $(1,c) \in N$ and set $t = 1, a = 0$, choose $s \in \Gm$ to be not a $k$-th root of unity, and $b = c (s^k-1)^{-1}$. Then the commutator $[(s,a), (t,b)] = (1,c)$.

(\ref{ext Aff: normal})
Let $N' \trianglelefteq G$ and assume that there is $g \in N' \setminus Z(G)$. Then there exists $h \in G$ such that $(1,0) \neq [h,g] \in N$ by (\ref{ext Aff: commutators}). Moreover $[h,g]g = hgh^{-1} \in N'$ and thus $[h,g] \in N'$ as well.
Thus there is a non-identity element $[h,g] \in N \cap N'$.
Observe that any two non-identity elements $(1,a), (1,b) \in N$ are conjugate: to have $(1,b) = (c,0)(1,a)(c,0)^{-1} = (1,c^k a)$ just pick $c$ as a $k$-th root of $\frac{b}{a}$. The conjugation orbit of $[h,g]$ does not leave the normal subgroup $N'$ and so $N \subseteq N'$.

(\ref{ext Aff: complements})
Let $K$ be any complement of~$N$ in~$G$ and let $H = \{(s,0) \mid s \in \Gm\}$ be the ``standard'' complement. Then the quotient $G \to \Gm$ restricts to an isomorphism on both $K$ and $H$, giving an isomorphism $H \to K$. Since the quotient map takes $(s,a)$ to $s$, this isomorphism maps each $(s,0) \in H$ to some $(s, \alpha(s)) \in K$, and $\alpha(s)$ is uniquely determined by $s$. In other words, $\alpha : \Gm \to \Ga$ is a definable function, and $K = \{(s,\alpha(s)) \mid s \in \Gm\}$. The fact that $K$ is a subgroup implies that for all $s,t \in \Gm$: $\alpha(st) = s^k \alpha(t) + \alpha(s)$. By taking $s,t$ to be algebraically independent over the field of definition of $K$ and writing $\alpha(st)=\alpha(ts)$, we see $s^k \alpha(t) + \alpha(s) = t^k \alpha(s) + \alpha(t)$. Hence $\alpha(s)(1-t^k) = \alpha(t)(1-s^k)$, or
\[\alpha(s)/(1-s^k) = \alpha(t)/(1-t^k).\]
Independence of $s,t$ implies that the ratio on the left is independent of $s$, and hence constant. So $\alpha(s)/(1-s^k) \eqqcolon w$ is constant in $s$, or
\[\alpha(s) = w(1-s^k).\]
Thus, $K = \Set{ (s, w(1-s^k)) \mid s \in \Gm }$ but this is conjugate to $H$ via~$(1,-w) \in N$.
\end{proof}
\begin{lemma}
\label[lemma]{lemma: semidirect QAut}
    Let $G = \Gm \ltimes_\theta \Ga$ and let  $\Phi \le G \times G$ be a definable and connected quasi-automorphism. Then $\Phi$ is the graph of a definable automorphism of $G$.
\end{lemma}
\begin{proof}
    Denote the projections of $\Phi \le G \times G$ to the left and right factors by $\pi_L$ and $\pi_R$, and denote $N_L = \ker \pi_L$ as well as $N_R = \ker \pi_R$.

    Since $\Phi \le G \times G$ is connected and quasi-isomorphic to $G$, it is itself a semidirect product $\Gm \ltimes_{\tilde \theta} \Ga$: first note that $\Phi$ has a $\Gm$ quotient with a one-dimensional kernel $K$ (first project onto the left $G$ factor via $\pi_L$, then from $G$ to $\Gm$). If $K$ is disconnected, $\Phi / K^0$ is a finite extension of $\Gm$, and $\Phi / K^0$ is connected (since $\Phi$ is) and isomorphic to $\Gm$ (since finite extensions of $\Gm$ cannot be isomorphic to $\Ga$ or to an elliptic curve). Also, $\pi_L$ maps $K^0$ into the kernel $\Ga$ of the quotient map $G \to \Gm$; it follows that $K^0$ is a finite extension of $\Ga$, and it is connected, hence isomorphic to $\Ga$. So there is an exact sequence $1 \to \Ga \to \Phi \to \Gm \to 1$, and $\Phi$ is affine as a subgroup of $G \times G$, hence linear. The result then follows from \cite[Section 19.1]{Humphreys}.

    Observe that for each of these projections $\pi:\Phi \to G$ we have that $\pi^{-1}(Z(G))$ is a finite normal subgroup (since $\ker \pi$ is finite), hence contained in $Z(\Phi)$; and $\pi(Z(\Phi))$ is a finite normal subgroup, hence contained in $Z(G)$. Surjectivity of $\pi$ then implies that $\pi(Z(\Phi)) = Z(G)$. Also, $\ker\pi \le Z(\Phi)$. So $\ker \pi = \ker[Z(\Phi) \overset{\pi}{\twoheadrightarrow} Z(G)]$ is the kernel of a homomorphism between cyclic groups, and it has order $\frac{|Z(\Phi)|}{|Z(G)|}$; but a cyclic group has at most one subgroup of any given order. Therefore $\ker \pi_L = \ker \pi_R$.

    It follows that $\pi_R \circ \pi_L^{-1}: G \to G$ is a well-defined homomorphism: the $\pi_L$-preimage of each $g \in G$ is a coset of $N_L=N_R$, which $\pi_R$ maps to a single element of $G$. It is an automorphism, since $\pi_L \circ \pi_R^{-1}$ is its inverse. These maps are definable, since $\Phi$ and the coordinate projections are both definable. Now we have
    \begin{align*}
      &\hphantom{{}={}} \{(x,y) \in G\times G \mid y = \pi_R \circ \pi_L^{-1}(x)\} \\
      &{}= \{(x,y)\in G \times G \mid \exists z\in G: (x,z) \in \pi_L^{-1}(x), y=\pi_R(x,z)=z\} \\
      &{}= \{(x,y) \in G \times G \mid (x,y) \in \pi_L^{-1}(x)\} = \Phi.
      \qedhere
    \end{align*}
\end{proof}

\begin{corollary} \label{lemma: Aut Aff}
Let $G = \Gm \ltimes_\theta \Ga$ be a connected, finite extension of $\Aff$ and $\Phi \le G \times G$ a quasi-automorphism. Consider the canonical quotient map $\pi\colon G \times G \to \Aff \times \Aff$, with kernel $Z(G) \times Z(G)$. Then the connected component of the identity in $\pi(\Phi)$ is the graph of a definable group automorphism $\varphi \in \Aut_{F}(\Aff)$. %
\end{corollary}
\begin{proof}
    The definable subgroup $\pi(\Phi) \le \Aff \times \Aff$ is two dimensional and it projects onto each coordinate of the product with finite kernel. Hence it is a quasi-automorphism, and the previous lemma applies to its identity component.
\end{proof}

\begin{lemma} \label{lemma: Aut G}
Let $\varphi \in \End_{F}(G)$ be a definable group epimorphism with finite kernel. Then there exist $b \in \Ga$, $c \in \Gm$ and $r \in \mathbb{Z}$ such that $\varphi(s,a) = (s^{p^r}, c a^{p^r} + b - \theta_{s^{p^r}}(b))$.
\end{lemma}

\begin{proof}
The subgroup $H = \Set{ (s,0) \in G }$ is a complement of $N = [G,G]$ in the semidirect product~$G$. The epimorphism $\varphi$ maps $H$ to some other complement $H'$ of~$N$. By \Cref{lemma: ext Aff}~(\ref{ext Aff: complements}) all these complements are conjugate and thus we may assume that $\varphi(H) \subseteq H$ up to an inner automorphism.

Because $\varphi(H) \subseteq H$ and $\varphi(N) \subseteq N$, we may consider $\alpha = \varphi \restriction_H \in \End_{F}(H)$ and $\beta = \varphi \restriction_N \in \End_{F}(N)$. \Cref{lemma: Def hom} shows that $\alpha(s) = s^{n/p^u}$ for some $n, u \ge 0$, and $\beta(a) = \sum_{i=0}^m c_i a^{p^i/p^v}$ is a $p$-Laurent polynomial with $c_m \neq 0$ and $v \ge 0$. We have $\varphi(s,a) = (\alpha(s),\beta(a))$ and the semidirect product structure imposes the relation
\begin{gather}
  \beta(\theta_s(a)) = \theta_{\alpha(s)}(\beta(a)).
\end{gather}
Hence $\beta(s^k a) = s^{kn/p^u} \beta(a)$. It follows that $\beta$ is homogeneous of degree $n/p^u = p^m/p^v$. Thus, with $r = m - v \in \ZZ$ this shows $\varphi(s,a) = (s^{p^r}, c a^{p^r})$ with $c \in F^\times$.
The above deduction applies after a suitable inner automorphism in $G$ which arranges $\varphi(H) = H$. The element to conjugate by can be chosen as $(1,b) \in [G,G]$ and results in a shift in the $\Ga$-coordinate:
\[
  (1,b) (s^{p^r}, c a^{p^r}) (1,b)^{-1} = (s^{p^r}, c a^{p^r} + b(1 - s^{k p^r})) = (s^{p^r}, c a^{p^r} + b - \theta_{s^{p^r}}(b)),
\]
as desired.
\end{proof}

\subsection{Recovering an affine group}
\label{sec: Aff recover}

We study geometric configurations of the following kind: $(x, y, z)$ and $(x', y', z')$ are two group triples over $F$ of connected groups, such that the unprimed elements have rank $2$ and the primed elements have rank $1$. We assume each primed element is algebraic over the corresponding unprimed one, and that the group of the triple $(x',y',z')$ is not of exponent $p$ (and hence is not $G_a$). Further, there is a group triple $(a,b,c)$ for $\Ga$, and elements $q, \tilde{b}$, each of rank $1$, satisfying: 
\begin{itemize}
    \item $a$ is algebraic over $x$, $\tilde{b}$ is algebraic over $y$, and $c$ is algebraic over $z$.
    \item The triple $x',\tilde{b},b$ is a circuit. That is, any two of its elements are algebraically independent, but the triple has rank $2$.
    \item $x \in \acl(x',a)$, $y\in \acl(y',\tilde{b})$, and $z \in \acl(z', c)$.
    \item the triples $b, q, y'$ and $\tilde{b}, q, z'$ are circuits. In particular, $x',y',z',\tilde{b},b,q$ form a group configuration.
\end{itemize}
This is depicted in \Cref{fig: Aff conf}, together with some extra points needed in \Cref{thm: Ga transfer}.
\begin{figure}
\begin{tikzpicture}[xscale=1.4, yscale=1.4]
  \definecolor{myblue}{HTML}{0057FF}
  \definecolor{myred}{HTML}{E0002A}
  \definecolor{mypurple}{HTML}{8A00E6}
  \definecolor{mygray}{HTML}{A8A8A8}
  \definecolor{myhighlight}{HTML}{C9A227}

  \pgfdeclarelayer{behind background}
  \pgfsetlayers{behind background,background,main}
  \tikzset{
    on behind background layer/.style={
      execute at begin scope={\begin{pgfonlayer}{behind background}},
      execute at end scope={\end{pgfonlayer}}
    }
  }

  \begin{scope}[on background layer, every node/.style={draw, circle, ultra thick, inner sep=10pt, mygray}, every path/.style={ultra thick, mygray}]
    \node (x) at (-3,-2) {};
    \node (w) at (-1,-2) {};
    \node (y) at (1,-2) {};
    \node (z) at (3,-2) {};
    
    \node (t1) at (0,1) {};
    \node (t2) at ($(x)!0.40!(t1)$) {};
    \node (t3) at ($(z)!0.40!(t1)$) {};

    \begin{scope}[transform canvas={xshift=-11pt}, shorten >=4pt, shorten <=-11pt]
      \draw (t1) -- ($(t1)!0.30!(t2)$);
      \draw[dashed] ($(t1)!0.415!(t2)$) -- ($(t1)!0.575!(t2)$);
      \draw ($(t1)!0.68!(t2)$) -- (t2);
      \draw (t2) -- (x);
    \end{scope}

    \begin{scope}[transform canvas={xshift=11pt}, shorten >=4pt, shorten <=-11pt]
      \draw (t1) -- ($(t1)!0.30!(t3)$);
      \draw[dashed] ($(t1)!0.415!(t3)$) -- ($(t1)!0.575!(t3)$);
      \draw ($(t1)!0.68!(t3)$) -- (t3);
      \draw (t3) -- (z);
    \end{scope}

    \draw (x) -- ($(x)!0.3!(w)$);
    \draw[dashed] ($(x)!0.4!(w)$) -- ($(x)!0.6!(w)$);
    \draw ($(x)!0.7!(w)$) -- (w);
    \draw (w) -- (y) -- (z);

    \def\eyelash{4pt}
    \begin{scope}[on background layer]
      \foreach \a in { 45, 90, 135 } { \draw (w.\a) -- ++(\a:\eyelash); }
      \foreach \a in { 45, 90, 135 } { \draw (y.\a) -- ++(\a:\eyelash); }
      \foreach \a in { 280, 325, 370 } { \draw (t2.\a) -- ++(\a:\eyelash); }
      \foreach \a in { 160, 205, 250 } { \draw (t3.\a) -- ++(\a:\eyelash); }
    \end{scope}
  \end{scope}

  \begin{scope}[every node/.style={fill, circle, inner sep=2.5pt, outer sep=0pt, myblue}, every path/.style={ultra thick, myblue}]
    \node[above=1pt of x.center] (x') {};
    \node[above=1pt of w.center] (w') {};
    \node[above=1pt of y.center] (y') {};
    \node[above=1pt of z.center] (z') {};
    \node[above=1pt of t1.center] (t1') {};
    \node[above=1pt of t2.center] (t2') {};
    \node[above=1pt of t3.center] (t3') {};

    \begin{scope}[on background layer]
      \draw (t1') -- ($(t1')!0.375!(t2')$);
      \draw[dashed] ($(t1')!0.425!(t2')$) -- ($(t1')!0.625!(t2')$);
      \draw ($(t1')!0.675!(t2')$) -- (t2');
      
      \draw (t1') -- ($(t1')!0.375!(t3')$);
      \draw[dashed] ($(t1')!0.425!(t3')$) -- ($(t1')!0.625!(t3')$);
      \draw ($(t1')!0.675!(t3')$) -- (t3');

      \draw (t2') -- (x');
      \draw (x') -- ($(x')!0.3!(w')$);
      \draw[dashed] ($(x')!0.4!(w')$) -- ($(x')!0.6!(w')$);
      \draw ($(x')!0.7!(w')$) -- (w');
      \draw (w') -- (y') -- (z') -- (t3');
    \end{scope}

    \def\eyelash{2pt}
    \begin{scope}[on background layer]
      \foreach \a in { 45, 90, 135 } { \draw (w'.\a) -- ++(\a:\eyelash); }
      \foreach \a in { 45, 90, 135 } { \draw (y'.\a) -- ++(\a:\eyelash); }
      \foreach \a in { 280, 325, 370 } { \draw (t2'.\a) -- ++(\a:\eyelash); }
      \foreach \a in { 170, 215, 260 } { \draw (t3'.\a) -- ++(\a:\eyelash); }
    \end{scope}
  \end{scope}

  \begin{scope}[every node/.style={fill, circle, inner sep=2.5pt, outer sep=0pt, myred}, every path/.style={ultra thick, myred}]
    \node[below=1pt of x.center] (x'') {};
    \node[below=1pt of w.center, mygray] (w'') {};
    \node[below=1pt of y.center, mygray] (y'') {};
    \node[below=1pt of z.center] (z'') {};
    \node[below=1pt of t1.center, mygray] (t1'') {};
    \node[below=1pt of t2.center, mygray] (t2'') {};
    \node[below=1pt of t3.center, mygray] (t3'') {};

    \begin{scope}[on background layer]
      \draw[name path=Aff-Ga-w, mypurple] plot[smooth, tension=0.6] coordinates { (x') ($(x')!0.65!(w') + (0, 0.4)$) (w'') ($(w'')+(0.5, -1)$) };
      \draw[name path=Aff-Ga-y, mypurple] plot[smooth, tension=0.6] coordinates { (x') ($(x')!0.65!(y') + (0, 0.4)$) (y'') ($(y'')+(0.5, -1)$) };

      \draw[name path=Ga, save path=\GaPath, draw=none] plot[smooth, tension=0.6] coordinates { (x'') ($(w'')+(0.5, -1)$) ($(y'')+(0.5, -1)$) (z'') };

      \path[name intersections={of=Aff-Ga-w and Ga, by={w'''}}];
      \path[name intersections={of=Aff-Ga-y and Ga, by={y'''}}];

      \draw[dash pattern=on 35pt
        off 7pt on 3pt off 3pt on 3pt off 3pt on 3pt off 3pt on 3pt off 3pt on 3pt off 7pt
        on 300pt][use path=\GaPath];
    \end{scope}

    \node (w'''n) at (w''') {};
    \node (y'''n) at (y''') {};

    \node (t1''') at ($(0, -5)$) {};
    \node (t2''') at ($(x)!0.60!(t1''')$) {};
    \node (t3''') at ($(z)!0.60!(t1''')$) {};

    \begin{scope}[on background layer]    
      \draw (t1''') -- ($(t1''')!0.35!(t2''')$);
      \draw[dashed] ($(t1''')!0.40!(t2''')$) -- ($(t1''')!0.60!(t2''')$);
      \draw ($(t1''')!0.65!(t2''')$) -- (t2''');
      
      \draw (t1''') -- ($(t1''')!0.35!(t3''')$);
      \draw[dashed] ($(t1''')!0.40!(t3''')$) -- ($(t1''')!0.60!(t3''')$);
      \draw ($(t1''')!0.65!(t3''')$) -- (t3''');

      \draw (t2''') -- (x'');
      \draw (t3''') -- (z'');
    \end{scope}

    \def\eyelash{2pt}
    \begin{scope}[on background layer]
      \foreach \a in { -55, -100, -145 } { \draw (w'''n.\a) -- ++(\a:\eyelash); }
      \foreach \a in { -35, -80, -125 } { \draw (y'''n.\a) -- ++(\a:\eyelash); }
      \foreach \a in { 350, 395, 440 } { \draw (t2'''.\a) -- ++(\a:\eyelash); }
      \foreach \a in { 90, 135, 180 } { \draw (t3'''.\a) -- ++(\a:\eyelash); }
    \end{scope}
  \end{scope}

  \begin{scope}[on background layer, every node/.style={fill, circle, inner sep=2.5pt, outer sep=0pt, mypurple}, every path/.style={ultra thick, mypurple}]
    \draw[name path=Gm-Ga-w-1] plot[smooth, tension=0.6] coordinates { (w'') ($(w'')!0.2!(z') + (0, -0.7)$) ($(w'')!0.6!(z') + (0, -0.7)$) (z') };
    \draw[name path=Gm-Ga-w-2] plot[smooth, tension=0.6] coordinates { (w') ($(w')!0.5!(w''') + (0.3, 0)$) (w''') };
    \path[name intersections={of=Gm-Ga-w-1 and Gm-Ga-w-2, by={qw}}];
    \node (qwn) at (qw) {};

    \draw[name path=Gm-Ga-y-1] plot[smooth, tension=0.6] coordinates { (y'') ($(y'')!0.2!(z') + (0, -0.3)$) ($(y'')!0.6!(z') + (0, -0.3)$) (z') };
    \draw[name path=Gm-Ga-y-2] plot[smooth, tension=0.6] coordinates { (y') ($(y')!0.5!(y''') + (0.3, 0)$) (y''') };
    \path[name intersections={of=Gm-Ga-y-1 and Gm-Ga-y-2, by={qy}}];
    \node (qyn) at (qy) {};
  \end{scope}

  \begin{scope}[every edge/.style={draw, ->, very thick}, every node/.style={inner sep=2pt, outer sep=0pt}]
    \node[myred] (x''_label) at ($(x''.center) + (-4ex,-1ex)$) {\bm{$a$}};
    \draw[myred] (x''_label) edge (x'');
    \node[myred] (y''_label) at ($(z''.center) + ( 4ex,-1ex)$) {\bm{$c$}};
    \draw[myred] (y''_label) edge (z'');
    \node[mygray] (w''_label) at ($(w''.center) + ( 4.5ex,7ex)$) {\bm{$\tilde{d}$}};
    \draw[mygray] (w''_label) edge[bend left=40] (w'');
    \node[mygray] (z''_label) at ($(y''.center) + (-5ex,7ex)$) {\bm{$\tilde{b}$}};
    \draw[mygray] (z''_label) edge[bend right=40] (y'');

    \node[myblue] (x'_label) at ($(x'.center) + (-3.8ex,1ex)$) {\bm{$x'$}};
    \draw[myblue] (x'_label) edge (x');
    \node[myblue] (z'_label) at ($(z'.center) + ( 4.2ex,1ex)$) {\bm{$z'$}};
    \draw[myblue] (z'_label) edge (z');
    \node[myblue] (w'_label) at ($(w''.center) + (0ex,6ex)$) {\bm{$w'$}};
    \draw[myblue] (w'_label) edge (w');
    \node[myblue] (y'_label) at ($(y''.center) + (0ex,6ex)$) {\bm{$y'$}};
    \draw[myblue] (y'_label) edge (y');

    \node[myred] (w'''_label) at ($(w'''n.center) + (-0.5ex,-3.5ex)$) {\bm{$d$}};
    \draw[myred] (w'''_label) edge (w'''n);
    \node[myred] (y'''_label) at ($(y'''n.center) + (-5ex,-3.5ex)$) {\bm{$b$}};
    \draw[myred] (y'''_label) edge[bend right=20] (y'''n);

    \node[mypurple] (qw_label) at ($(qwn.center) + (3.6ex,1ex)$) {\bm{$r$}};
    \draw[mypurple] (qw_label) edge (qwn);
    \node[mypurple] (qy_label) at ($(qwn.center) + (7ex,0.85ex)$) {\bm{$q$}};
    \draw[mypurple] (qy_label) edge[bend right=15] (qyn);
  \end{scope}

\end{tikzpicture}
\caption{The configuration featuring in the theorems of \Cref{sec: Aff recover} and in \Cref{thm: main}. The points in red are part of a $\Ga$ plane; the points in blue are part of an $H$-plane, for $H$ an algebraic group which is not of exponent $p$ (hence not $\Ga$). The gray circle containing $x'$ is a rank-2 point called $x$; similarly we have rank-2 points $w,y,z$. Half-edges emanating from points indicate that some lines are omitted (these are the lines necessary to obtain group configurations, and for \Cref{thm: quasi-aut} to apply to $w$ and to $d$ in \Cref{thm: Ga transfer}).
}
\label{fig: Aff conf}
\end{figure}
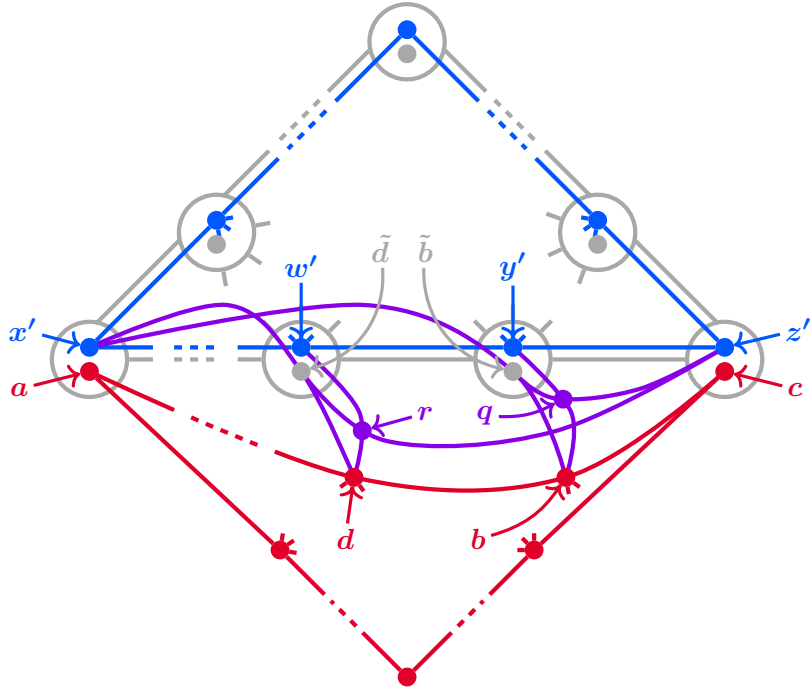
\begin{theorem} \label{thm: Aff conf}
    In the situation above, $(x,y,z)$ is a group triple for a group $G\simeq \Gm \ltimes \Ga$, where the action of $\Gm$ on $\Ga$ is given by $s.a = s^k a$ for some $k \neq 0$. In particular, $G$ is quasi-isomorphic to $\Aff$. Further, the group $H$ of $(x',y',z')$ is quasi-isomorphic to the quotient $\Gm$ of $G$, and the algebraic dependence of $a$ over $x$ induces a quasi-isomorphism between $\Ga$ and the kernel $K$ of the quasi-epimorphism $G \to \Gm$ obtained from \Cref{thm: quasi-iso}.
\end{theorem}
\begin{remark}
    These configurations correspond to the semidirect product decomposition $\Aff \simeq \Gm \ltimes \Ga$ as follows: write elements of $\Aff$ as pairs $x=(x',a)$ with $x' \in \Gm$ and $a \in \Ga$. The ``additive coordinate'' of $y = x^{-1} z$ depends on the additive coordinates of both $x,z$, but on the multiplicative coordinate of $x$ only. Writing $z=(z',c)$, the additive coordinate $\tilde{b}$ of $y$ is in fact $\tilde b=x'^{-1}\cdot b$, where $b=c-a$, and this multiplicative relation is encoded by the additional group configuration (for which we have added the additional point $q$). The point of the theorem is that these relations are not given, and we want to prove that they hold. This will almost be accomplished: $G$ is only isogenous to $\Aff$ and not necessarily isomorphic; and we can say slightly less about the ``additive coordinate'' points $a,\tilde{b},c$. (In particular, it turns out they are not necessarily interalgebraic with ``the additive coordinates'' of $x,y,z$ with respect to a splitting of $G$.) The proof is slightly roundabout: we first prove that there is an exact sequence $1 \to \Ga \to G \to L \to 1$, then prove that this extension is not central. Semidirect products $\Gm \ltimes \Ga$ are then the only possibilities for $G$.

    The reader may ask why we do not apply the Field Configuration Theorem, instead. Roughly, that theorem states that a faithful action of a $2$-dimensional group on a $1$-dimensional set is essentially the action of $\Aff$ on $\Ga$, and certain configurations imply the existence of such an action (see \cite[Cor.~6.6]{Bays} for an exact statement). A main difficulty is to obtain faithfulness of the action, or at least finiteness of the kernel, purely in terms of conditions on transcendence degrees. The field configuration does not achieve this on its own: it has linear representations, which, by taking linear combinations of independent transcendentals, give rise to algebraic representations; in any such algebraic representation, the group of every group configuration is a power of $\Ga$, and $\Aff$ cannot appear. (In \cite[Thm.~6.1]{Bays}, for instance, the condition guaranteeing faithfulness is phrased in terms of canonical bases: in the notation there, it is the assumption that $\acl^{\mathrm{eq}}(a)=\overline{\Cb}(xy/a)$, and similarly for $bzy$ and $czx$.)
\end{remark}
\begin{proof}
    From \Cref{thm: quasi-iso} we obtain a quasi-epimorphism between $G$ and $H$. For convenience, identify this with a morphism $\pi:G \to L$ from $G$ onto a quotient $L$ of $H$ by a finite normal subgroup. We may assume $\pi$ is a morphism of algebraic varieties by replacing $L$ with a Frobenius twist if necessary: see \Cref{lemma: Def hom}.

    If $\ker(\pi)$ happens to be disconnected, we perform the following replacement process: the connected component of the identity in $\ker(\pi)$ is normal, since conjugation in $G$ induces an automorphism of the kernel and automorphisms preserve the connected component of the identity. Denote $K = \ker(\pi)^0$ and temporarily denote $\tilde{L} = G / K$. Observe that $\tilde{L}$ must be connected, as $\tilde{L}$ is a continuous image of $G$; and $G$ is connected by assumption. Note that $\tilde{L}$ has $L$ as a quotient by a finite normal subgroup (which we can identify with the group of components of $\ker(\pi)$), so $\tilde{L}$ is still quasi-isomorphic to $H$. Now replace $L$ by $\tilde{L}$, and $\pi$ by the quotient map $G \to G/K$.

    In short, we have a short exact sequence
    \[1 \to K \to G \overset{\pi}{\to} L \to 1\]
    of connected algebraic groups, and $L$ is quasi-isomorphic to $H$.

    Similarly to the proof of \Cref{thm: quasi-iso}, we obtain relations $\Rel 1$, $\Rel 2_s$, and $\Rel 3$, definable over $F$, and where $\Rel 2_s$ depends on an additional parameter $s$ with values in $L$.
    Explicitly, $\oRel 1 \subseteq G \times G_a$ is obtained by taking an $F$-definable affine open $U \subseteq G \times G_a$ that contains $(x,a)$. In it, consider the subset $C \subset U$ of all pairs satisfying the same equations as $(x,a)$ (i.e. $C=Z(I_F((x,a)))$ is the unique minimal $F$-definable Zariski-closed subset of $U$ containing $(x,a)$). Then take $\Rel 1$ to be the closure of $C$ within $G \times G_a$. The relation $\oRel 3 \subseteq G \times G_a$ is defined similarly. For $\oRel 2 \subseteq G \times L \times G_a$, perform the analogous process with an $F$-definable affine open of $G \times L \times G_a$ containing $(y, \pi(x), b)$. For $s \in L$, $\oRel 2_s$ is notation for $\{(u,c) \in G \times G_a \mid (u,s,c) \in \oRel 2\}$.

    \begin{claim}[``Two of three'']
    If any two of
    \[\begin{cases}
    u \Rel 1 d \\ 
    v \Rel 2_{\pi(u)} e \\
    (u\cdot_G v) \Rel 3 d +_{G_a} e
    \end{cases}\]
    hold then so does the third.
    \end{claim}
    \begin{proof}
    In what follows we suppress $\cdot_G$ and $+_{G_a}$ from the notation and just write, for example, $uv$ or $d+e$.
    Observe that the statement holds generically. In detail:
    \begin{enumerate}
    \item If $u,v \in G$ are generic and independent over $F$, and $d,e \in G_a$ satisfy $u \Rel 1 d$ and $v \Rel 2_{\pi (u)} e$, then $uv \Rel 3 d + e$.
    \item If $u,w \in G$ are generic and independent over $F$, and $d,f \in G_a$ satisfy $u \Rel 1 d$ and $w \Rel 3 f$, then $u^{-1} w \Rel 2_{\pi(u)} f - d$.
    \item If $v,w \in G$ are generic and independent over $F$, and $e,f \in G_a$ satisfy $v \Rel 2_{\pi(w v^{-1})} e$ and $w \Rel 3 f$, then $w v^{-1} \Rel 1 f - e$.
\end{enumerate}
In all three cases, the result follows from \Cref{lemma: types}. Now each of the three desired implications is of the following general form: there are closed subsets $A,B$ of some variety, and a Zariski-continuous function $f: A \times B \to C$, where $C$ is a third variety; for a dense subset $U \subset A\times B$ we have $f(U) \subset C'$, where $C' $ is a closed subset of $C$. Under these conditions, continuity implies $f(\overline U) \subset \overline{f(U)} \subset \overline{C'} = C'$. 

Concretely, for the second implication, $A = \oRel 1$, $B = \oRel 3$, the dense subset $U \subset A \times B$ is the set of tuples $((u,d),(w,f)) \in A \times B$ such that $u,w$ are generic and independent over $F$, $C = G \times L \times G_a$, $f((u,d),(w,f)) = (u^{-1} w, \pi(u), f - d)$, and $C' = \oRel 2$. The two other implications are entirely analogous.
\end{proof}

\begin{claim}
    Each of $\oRel 1$ and $\oRel 3$ has dimension $2$. There exists a finite subset of $L$ such that for all $s$ outside this finite subset, $\dim \oRel 2_s = 2$.
\end{claim}
\begin{proof}
For $\oRel 1$ and $\oRel 3$ the projection to the $G$-coordinate is dominant and has finite fibers, so the claim follows.

For $\oRel 2$, first note that for $s \in L$ generic over $F$ we have $\dim(\oRel 2_s)=2$. This is because it holds for the generic element $s = \pi(x)$: we have $\RM(y,b/Fs)=\RM(y/Fs)=2$, since $y,s$ are independent over $F$ and $b$ is algebraic over $y,s$ by assumption ($x',\tilde b, b$ is a circuit, and $\tilde b$ algebraic over $y$). 

Now let $p$ be the projection of $\oRel 2 \subset G \times L \times G_a$ to $L$. This is a morphism of algebraic varieties. The subset $\{t \in L \mid \dim(p^{-1}(t)) = 2\}$ is constructible and contains the generic element $s$, so it contains an open $V \subseteq L$; the complement of $V$ is finite since $L$ is irreducible of dimension $1$.
\end{proof}

\begin{claim}\label[claim]{claim:U in L x L}
    There exists a dense open subset $U$ of $L \times L$, definable over $F$, such that for all $(s,t) \in U$ the cosets of $K$ given by
    \[C_1 = \pi^{-1}(s),\ C_2 = \pi^{-1}(t), \text{ and } C_3 = \pi^{-1}(st)\]
    satisfy 
    \[\dim(\oRel 1 \cap (C_1 \times G_a)) = \dim(\oRel 2_{s} \cap (C_2 \times G_a)) = \dim(\oRel 3 \cap (C_3 \times G_a)) = 1.\]
\end{claim}
\begin{proof}
Consider the projection from $\oRel 1 \subset G \times \Ga$ to $L$ (that is, the composition of $\pi$ on the projection $\oRel 1 \to G$). This is a projection of irreducible varieties, so over a nonempty open subset $V$ of $L$ the fibers have dimension $\dim \oRel 1 - \dim L = 1$. Intersecting $V$ with the open subset of $L$ of the previous claim, we also obtain that $\dim \oRel 2_s = 2$ for all $s \in V$. In the same way, given $s \in V$, the set of elements $t \in L$ such that $\dim(\oRel 2_{s} \cap (\pi^{-1}(t) \times G_a)) = 1$ is a nonempty open. So the constructible set 
\[U = \{(s,t)\in L \times L \mid s \in V, \dim(\oRel 2_{s} \cap (\pi^{-1}(t) \times G_a)) = 1\}\]
has Morley rank $2$, and hence contains a nonempty open subset of $L \times L$.

Now let $(s,t) \in U$ and denote $C_1 = \pi^{-1}(s)$, $C_2 = \pi^{-1}(t)$, $C_3 = \pi^{-1}(st)$. Fix $(u,d) \in \oRel 1 \cap (C_1 \times \Ga)$, and observe that the ``two of three'' property implies that multiplication by $(u,d)$ induces a bijection between $\oRel 2_s \cap (C_2 \times \Ga)$ and $\oRel 3 \cap (C_3 \times \Ga)$. Hence $\dim \oRel 3 \cap (C_3 \times \Ga) = \dim \oRel 2_s \cap (C_2 \times \Ga) = 1$.
\end{proof}

\begin{claim}
    Let $(s,t) \in U$, where $U \subset L \times L$ is as in the previous claim. Fix $u\in \pi^{-1}(s)$ and $v \in \pi^{-1}(t)$, and choose $d,e \in G_a$ such that $u \Rel 1 d$ and $v \Rel 2_{s} e$. Then the relations
    $\oRel 1',\oRel 2', \oRel 3' \subset K \times G_a$ given by:
    \begin{align}
    k \Rel 1' m & \iff  k\cdot u \Rel 1 m + d\\
    k \Rel 2' m &\iff  u^{-1} k u v \Rel 2_s m + e\\ 
    k \Rel 3' m &\iff k uv \Rel 3 m + (d+e).
    \end{align}
    are all equal to each other, and are a quasi-isomorphism.
\end{claim}
\begin{remark}
    The ``two of three'' property here means that, for instance, if $k_1 \Rel 1' m_1$ and $k_2 \Rel 2' m_2$ then $k_1 k_2 \Rel 3' m_1 + m_2$, and similarly in the other two cases. (It does not mean that if $k \Rel 1' m$ and $k \Rel 2' m$ then $k \Rel 3' m$ for the same $k,m$.)
\end{remark}
\begin{proof}
We apply the same kind of translation argument as in \Cref{thm: quasi-iso}. A straightforward verification shows that these relations satisfy the ``two of three'' property, and that for $e_K$ the neutral element of $K$ and $0 \in G_a$ we have $e_K \Rel 1' 0$, $e_K \Rel 2' 0$, and $e_K \Rel 3' 0$. As in \Cref{thm: quasi-iso} it follows that $\oRel 1' = \oRel 2' = \oRel 3'$ is a quasi-isomorphism between $K$ and $G_a$.
\end{proof}

\begin{corollary}
\label[corollary]{cor: Upsilon}
    There is a quasi-isomorphism $\Upsilon$ between $K$ and $\Ga$, definable over $F$, and a dense $F$-definable open subset $V \subset L$, such that the following hold:
    \begin{enumerate}
        \item If $g \in \pi^{-1}(V)$ then $\dim (\oRel 1 \cap (Kg \times \Ga)) = \dim(\oRel 3 \cap (Kg \times \Ga)) = 1$.
        \item If $g \in \pi^{-1}(V)$ and $g \Rel 1 d$ then
        $k \mathrel{\Upsilon} m \iff kg \Rel 1 m+d.$
        \item If $g \in \pi^{-1}(V)$ and $g \Rel 3 f$ then
        $k \mathrel{\Upsilon} m \iff kg \Rel 3 m+f.$
    \end{enumerate}
\end{corollary}
\begin{remark} This essentially means that the relations $\oRel 1'$ and $\oRel 3'$ of the previous claim do not depend on the choice of $s,t,u,v,d,e$. For $\oRel 2'$, the additional parameter makes any analogous statement somewhat tricky; but we do not need it.
\end{remark}
\begin{proof}
Let $U \subset L \times L$ be the open set of \Cref{claim:U in L x L}. Since $L$ is irreducible, so is $L \times L$, and therefore $U$ is dense in $L \times L$. So there exists an $F$-definable $s_0\in L$ in the projection of $U$ to the first coordinate. Fix an $F$-definable $g_0 \in \pi^{-1}(s_0)$, as well as an $F$-definable $d_0 \in G_a$ such that $g_0 \Rel 1 d$. Define $\Upsilon$ by 
\[k \Upsilon m \iff kg_0 \Rel 1 m+d_0.\]

Note that $\{t' \in L \mid (s_0,t')\in U\}$ is open in $L$. For any such $t'$ and any $h \in \pi^{-1}(t')$, let $f \in \Ga$ be an element satisfying $g_0h \Rel 3 f$. Then the relation $\Rel 3'$ defined by
\[ k \Rel 3' m \iff kg_0h \Rel 3 m + f \]
equals $\Upsilon$: by the ``two of three'' property we have that $h \Rel 2_{s_0} f - d_0$, and the previous claim applies. This shows that if $u$ is any element of
\[V'' \coloneqq \{u \in L \mid u=s_0 t', (s_0,t)\in U\}\]
satisfies that if $g\in \pi^{-1}(u)$ and $g \Rel 3 f$ then $k \mathrel\Upsilon m \iff kg \Rel 3 m + f$. Observe that $V''$ is an $F$-definable open subset of $L$.

Now fix an $F$-definable $u_0 \in V''$, $g_0 \in \pi^{-1}(u_0)$, and $f_0 \in \Ga$ such that $g_0 \Rel 3 f_0$. By the above argument we have $k \mathrel\Upsilon m \iff kg_0 \Rel 3 m+f_0$. The set
\[W = \{t \in L \mid (u_0 t^{-1}, t) \in U\}\]
is Zariski open over $F$: it is the preimage of $U$ under the morphism $L \to L \times L$ given by $t \mapsto (u_0 t^{-1}, t)$. Let $t \in W$ and let $s = u_0 t^{-1}$. Similarly to the above, we have that if $g \in \pi^{-1}(s)$ and $g \Rel 1 d$ then $k \mathrel\Upsilon m \iff kg \Rel 1 m + d$, for the same $\Upsilon$ as defined earlier.

Finally, define $V = V'' \cap \{u_0 t^{-1} \mid t \in W\}$.
\end{proof}

\begin{claim}
    The extension $1 \to K \to G \to L \to 1$ is not central.
\end{claim}
\begin{proof}
    Assume for a contradiction that the extension is central.

    We have $K \simeq \Ga$ since $K$ is a connected $1$-dimensional group (hence isomorphic to one of $\Gm$, $\Ga$, or an elliptic curve) and quasi-isomorphic to $\Ga$ via $\Upsilon$ (which rules out isomorphism with $\Gm$ or any elliptic curve).

    There is a rational section $\alpha: L \supseteq U \rightarrow  G$ for some dense open $U \subseteq L$ (see \cite[Ch.~VII, Sec.~1, Prop.~6]{Serre_alg_gps_and_class_fields}). A rational section $\alpha$ gives rise to a rational cocycle $\gamma$ whose cohomology class in $H^2_{\mathrm{rat}}(L, K)$ determines the extension. Concretely, $\gamma$ is the rational function
    \[\gamma: L \times L \supseteq U\times U \rightarrow K\]
    given by $\gamma(s_1,s_2) = \alpha(s_1 s_2) \alpha(s_2)^{-1} \alpha(s_1)^{-1}$ whenever $s_1,s_2 \in U$.
    It follows that if $k_1,k_2 \in K$ and $s_1,s_2 \in L$, we have
    \[(k_1 \cdot \alpha(s_1)) \cdot (k_2 \cdot \alpha(s_2)) = k_1 k_2 \alpha(s_1) \alpha(s_2) = k_1 k_2 \gamma(s_1,s_2)^{-1} \alpha(s_1s_2).\]

    Fix a section $\alpha$ as above, such that $\alpha$ is defined over $F$. By replacing $U$ with a smaller open set if necessary, we may assume it is contained in a set $V \subset L$ as in the previous claim.
    
    Observe that $\pi(z) \in U$, since $U$ is $F$-definable and $\RM(\pi(z)/F) = 1$ (for instance because $\pi(z)$ is interalgebraic with $z'$). We can write $z = \hat{c} \cdot \alpha(\pi(z))$ where $\hat{c} \in K$, and this defines $\hat{c}$ uniquely. Find $m \in \Ga$ such that $\alpha(\pi(z)) \Rel 3 m$. Then, since
    \[\hat{c} \cdot \alpha(\pi(z)) \Rel 3 (c - m) + m,\]
    we have $\hat{c} \mathrel{\Upsilon} c-m$.  %

    In exactly the same way, $x = \hat{a} \cdot \alpha(\pi(x))$ for some unique $\hat{a} \in K$. Find $n \in \Ga$ such that $\alpha(\pi(x)) \Rel 1 n$, and observe that
    \[\hat{a} \mathrel{\Upsilon} a - n.\]

    Choose $\hat{m},\hat{n} \in K$ satisfying $\hat{m} \mathrel{\Upsilon} m$ and $\hat{n} \mathrel{\Upsilon} n$. Then $\hat{a} + \hat{n} \mathrel{\Upsilon} a$ and $\hat{c} + \hat{m} \mathrel{\Upsilon} c$. So 
    \[\hat{c} - \hat{a} + (\hat{m}-\hat{n}) \mathrel{\Upsilon} b,\] 
    because $b=c-a$. Note that over $F$ we have that $\hat{m}$ is algebraic over $\pi(z)$ and $\hat{n}$ is algebraic over $\pi(x)$.

    Note that by our assumptions on the ranks, $x',y',z',\tilde{b},b,q$ form a group configuration. Further, $(x',y',z')$ is a group triple for $H$ which is interalgebraic with the group triple $(\pi(x),\pi(y),\pi(z))$ for $L$. 
    So (up to extending $F$ by transcendentals independent of everything in sight and passing to the algebraic closure,) by \Cref{thm: group config} and \Cref{cor: quasi-iso config} there exist $\hat{\tilde{b}}, \hat{b},\hat{q} \in L$, interalgebraic over $F$ with $\tilde b, b, q$ respectively, such that $\pi(x)\cdot \hat{b} = \hat{\tilde{b}}$.

    Now let $E$ be the algebraic closure of the field $F(\pi(y))$. Since $\pi(z) = \pi(x)\pi(y)$ and $\pi(x),\pi(y)$ are independent over $F$, we have that $\pi(x)$ and $\pi(z)$ are interalgebraic transcendentals over $E$. Also, $y$ and $\tilde{b}$ are interalgebraic over $E$, since over $F$ we have that $y$ is algebraic over $(y',\tilde{b})$, and $y'$ is interalgebraic with $\pi(y)$.

    To summarize, working over $E$ we have:
    \begin{itemize}
        \item Elements $\hat{m},\hat{n}\in K$ algebraic over $\pi(x)$.
        \item Elements $\hat{a},\hat{c} \in K$ such that $\hat{a}+\hat{n} \mathrel{\Upsilon} a$ and $\hat{c}+\hat{m} \mathrel{\Upsilon} c$.
        \item An element $\hat{b} \in L$ interalgebraic with $b=c-a$ such that $\pi(x)\hat{b}$ is interalgebraic with $y$.
    \end{itemize}

    Write $y = k\cdot \alpha(\pi(y))$ for $k\in K$, and note that $k$ is interalgebraic with $y$ (over $E$). The assumption that the extension is central implies 
    \[(\hat{a} \cdot \alpha(\pi(x))) \cdot (k\cdot \alpha(\pi(y))) = (\hat{a}k\cdot \gamma(\pi(x),\pi(y))^{-1})\alpha(\pi(x)\pi(y)),\]
    where $\pi(x)\pi(y)=\pi(z)$. In other words, it implies (now using additive notation for $K$) that $\hat{a} + k - \gamma(\pi(x),\pi(y)) = \hat{c}$, or $k = \hat{c} - \hat{a} + \gamma(\pi(x), \pi(y))$. Here $\gamma(\pi(x), \pi(y))$ is algebraic over $\pi(x)$. 

    We also have that $\hat c - \hat a + (\hat m - \hat n) \mathrel \Upsilon b$, since $b=c-a$. Since $(\hat m - \hat n)$ is algebraic over $\pi(x)$, also
    \[\hat \eta \coloneqq (\hat m - \hat n) - \gamma(\pi(x),\pi(y))\]
    is algebraic over $\pi(x)$. Now $k + \hat \eta \mathrel \Upsilon b.$ Find $\eta \in \Ga$ such that $\hat \eta \mathrel \Upsilon -\eta$, so that
    \[k \mathrel \Upsilon b + \eta.\]

    We now have that over $E$, $y$ is interalgebraic with $b + \eta$ (where $\eta$ is interalgebraic with $\pi(x)$) and also $y$ is interalgebraic with $\pi(x)\hat b$ (where $\hat b$ is interalgebraic with $b$). So we have two group triples
    \[(b, \eta, b+\eta) \text{ and } (\hat{b}, \pi(x), \pi(x)\hat{b}),\]
    which are interalgebraic over $E$, but one is a group triple for $\Ga$ and the other for $L$. Note that they do in fact have rank $2$ over $E$: $b, x, y'$ are independent over $F$, and $y'$ is interalgebraic over $F$ with $\pi(y)$, so $b,x$ are independent over $E$. By \Cref{thm: quasi-iso}, this gives a contradiction to our assumption that $L$ is not quasi-isomorphic to $\Ga$.
\end{proof}

This rules out the situation where $L$ is an elliptic curve: let $s \in L$ and lift $s$ to $g \in G$. Then conjugation by $g$ defines an automorphism of $K$, and the automorphism is independent of the lift. So \Cref{lem:action_of_quotient} shows that conjugating a fixed $m \in K$ by elements of $L$ and then applying a large enough power $k$ of the Frobenius defines a morphism of algebraic varieties $L \to \Frob^k(K)$. But if $L$ is an elliptic curve, any such morphism is constant (as $K \simeq \Ga$ is affine), so the conjugation action on $K$ is trivial and the extension must be central, contradiction.

We now have $L \simeq \Gm$, and it acts nontrivially on $\Ga$ by conjugation within $G$. This action is, in particular, by endomorphisms of $\Ga$; it follows that there must exist $k\in\ZZ$ such that the action satisfies $s.a = s^k \cdot a$ for all $s \in L$, $a \in K$. There is a unique extension $1 \to \Ga \to G \to \Gm \to 1$ of this form, the semidirect product: see for example \cite[Section 19.1]{Humphreys} (our $G$ is affine by \cite[Ch. 3, Sec. 7, Prop. 11]{Serre_alg_gps_and_class_fields}, hence linear). This proves \Cref{thm: Aff conf}.
\end{proof}

\begin{theorem} \label{thm: Ga transfer}
    Under the assumptions of the previous theorem, assume there are also elements $w\in G$, $w' \in H$, and $d \in G_a$, and an element $\tilde{d}$ of the field $\mathbb{K}$, such that:
    \begin{itemize}
        \item $w'$ and $\tilde{d}$ are algebraic over $w$, 
        \item $w$ is algebraic over $(w',\tilde{d})$,
        \item The triple $x',\tilde{d},d$ is a circuit (meaning that any two of its elements are algebraically independent, but the triple has rank $2$).
    \end{itemize}
    Further assume that there exist quasi-automorphisms $\Phi$ of $G$ and $\varphi$ of $G_a$ such that $d \bumpeq c - \varphi(a)$ and $w \bumpeq \Phi(x)^{-1}\cdot z$. Denote the quasi-isomorphism between $\Ga$ and $K$ constructed in the previous theorem by $\Upsilon$. Then $\Phi_{K} \coloneqq \Phi \cap (K \times K)$ is a quasi-automorphism of $K$, and the following diagram of quasi-isomorphisms is commutative up to $\bumpeq$:
    \[\xymatrix{
        K\ar[d]_\Upsilon\ar[r]^{\Phi_{K}} & K\ar[d]^\Upsilon \\
        \Ga\ar[r]^\varphi & \Ga.
    }\]
\end{theorem}
\begin{proof}
We may assume $\Phi$ is connected by passing to $\Phi^0$ if necessary. \Cref{lemma: semidirect QAut} shows $\Phi$ is the graph of an $F$-definable automorphism of $G$, so the equation $w \bumpeq \Phi(x)^{-1} z$ is an honest equality. 

As an automorphism of groups, $\Phi$ restricts to an automorphism of the derived subgroup $[G,G]=K$ (see \Cref{lemma: ext Aff}), so $\Phi_K = \Phi \cap (K \times K)$ is a quasi-automorphism of $K$ (in fact the graph of an $F$-definable automorphism).

Using the isomorphism $G\simeq \Gm \ltimes_\theta \Ga$, we write $x = (x_m, x_a)$, $w=(w_m, w_a)$, and $z = (z_m, z_a)$. It is convenient to also denote $\Phi(x) = (\hat x_m, \hat x_a)$. Observe that while $\hat x_m$ depends only on $x_m$, $\hat x_a$ may depend both on $x_a$ and on $x_m$. Specifically, by \Cref{lemma: Aut G} there exist $F$-definable $q_1 \in \Gm$, $q_2 \in \Ga$, and $r\in\ZZ$, such that $\hat x_a = q_1x_a^{p^r} + q_2 - \theta_{x_m^{p^r}}(q_2)$.

By assumption we have
\[w = (\hat x_m ^{-1} z_m, \ \hat x_m^{-1}(z_a - \hat x_a)).\] 
Therefore $z_a - \hat x_a$ is in the closure of $\{\hat x_m, w\}$, which equals the closure of $x_m, w$. From the explicit description of $\hat x_a$ above, we see that also $k\coloneqq z_a - q_1x_a^{p^r}$ is in the closure of $x_m, w$. We also have that $d$ is in the closure of $x_m, w$, because $\acl(x_m) = \acl(x')$.

Define $E = \overline{F(x_m, z_m)}$. Then in $E$ we have $\acl(x_a)=\acl(a)$, $\acl(z_a)=\acl(c)$, as well as $\acl(k)=\acl(x_m,w)=\acl(x',w)=\acl(d)$: 
\begin{itemize}
    \item To see $\acl(k)=\acl(x_m,w)$ over $E$ observe that $\acl(x_m)=\acl(\emptyset)$ and $\acl(w)=\acl(w_a)$, where $w_a=x_m^{-1} \cdot (k - q_2 + \theta_{x_m^{p^r}}(q_2))$, and $q_2$ is $F$-definable. 
    \item To see $\acl(d)=\acl(x',w)$ over $E$ observe that $\acl(d)=\acl(d,x',w')$ since $x',w'$ are in $E$, and $\tilde{d} \in \acl(d,x')$. So $\tilde{d},w'\in\acl(d)$ over $E$; but $w \in \acl(\tilde{d},w')$ over $F$, hence also over $E$.
\end{itemize}

Also note that $\RM(a,c/E)=\RM(x_a,z_a/E)=2$. So over $E$, $(\Phi_K(x_a), k, z_a)$ and $(\varphi(a),d,c)$ are two elementwise interalgebraic group triples. By \Cref{thm: quasi-iso} there exists an $E$-definable quasi-isomorphism $\Theta$ taking one to the other, up to $E$-definable translations. These are $\Ga$ group triples, so the translations generate finite normal subgroups of $\Ga$, and they may be eliminated by enlarging $\Theta$. In particular, 
\[z_a \mathrel\Theta c.\] We will show that also 
\[z_a \mathrel \Upsilon c,\]
up to an $E$-definable translation that can be eliminated by enlarging $\Upsilon$ by a finite subgroup.
Observe that this implies $\Theta \bumpeq \Upsilon$: first, it implies that $c \mathrel{\Theta \circ \Upsilon^{-1}} c$. Now, writing $\Theta \circ \Upsilon^{-1}$ as $\{(\alpha,\beta) \in \Ga \times \Ga \mid f(\alpha) = g(\beta)\}$ for appropriate $p$-polynomials $f,g\in \End^\mathrm{alg}_E(\Ga)$, we see $f(c)=g(c)$. But $f,g$ are in particular polynomials with coefficients in $E$, and $c$ is transcendental over $E$. So $f = g$. Hence $\Theta \circ \Upsilon^{-1} \supseteq \{(x,x)\mid x \in \Ga\}$. But the right-hand side of the inclusion (the identity quasi-automorphism) is of dimension $1$, so it must have finite index in $\Theta \circ \Upsilon^{-1}$, and the two are $\bumpeq$-equivalent.

It remains to show that $z_a \mathrel \Upsilon c$, up to an $E$-definable translation. In the notation of \Cref{thm: Aff conf} (and \Cref{cor: Upsilon}), we have $z \Rel 3 c$. Let $\sigma$ be the section of the quotient map $G \to \Gm$ that takes $s\in\Gm$ to $(s,0) \in G$. Then $z = z_a \sigma(z_m) \Rel 3 c$. There exists $\lambda \in E$ such that $\sigma(z_m) \Rel 3 \lambda$, since $\sigma$ and $\oRel 3$ are both $F$-definable, and $z_m \in E$. So $z_a \Upsilon c - \lambda$ by \Cref{cor: Upsilon} (and by the fact that $z_m$ is transcendental over $F$, hence in the open set $V$ of that corollary). The theorem follows.

\end{proof}

\section{Undecidability}

The affine group, through its group of quasi-automorphisms, connects the endomorphism ring of $\Gm$ to that of $\Ga$. We use this connection and the von~Staudt technique due to Evans--Hrushovski to give an existential definition of the field $\FF_p(t)$ inside $F(x; \Frob)$ with a distinguished element~$t$. This element is constructed first in $L_F(\Gm)=\QQ$ and then transferred to $L_F(\Ga)=F(x;\Frob)$. The following theorem due to Pheidas (for characteristic $\neq 2$) and Videla (for characteristic $2$) then yields undecidability.

\begin{theorem}[\cite{Pheidas,Videla}] \label{thm: Pheidas--Videla}
For each prime power $q$ the existential theory of the field $\FF_q(t)$ with a constant symbol for $t$ is undecidable.
\end{theorem}

The proof encodes the arithmetic of $\ZZ$ into the $t$-adic orders of elements of $\FF_q(t)$. %
The constructed existential formulas require $t$ to be available in the language and they are not uniform in the finite base field but require its cardinality~$q$ to be known in advance. %

\begin{proposition} \label{prop: Gm to Aff quasi-aut}
Consider quasi-automorphisms $\Phi \le G \times G$ and $\Phi' \le \Gm \times \Gm$ and a quasi-epimorphism $\Psi \le G \times \Gm$ such that $\Psi \circ \Phi \bumpeq \Phi' \circ \Psi$.
If $\Phi'$ is equivalent to the Frobenius $s \mapsto s^p$ then $\Phi$ is equivalent to the group automorphism $(s,a) \mapsto (s^p, c a^p + b - \theta_{s^p}(b))$ for constants $b \in \Ga$ and $c \in \Gm$.
\end{proposition}

\begin{proof}
Since $\Psi$ is a quasi-epimorphism of $G$ onto $\Gm$, its kernel $K = \Set{ (s,a) \in G \mid (s,a) \mathrel{\Psi} 1 }$ is quasi-isomorphic to~$[G, G] \simeq \Ga$. Indeed, $\Psi$ induces an epimorphism $\psi\colon G \to \Gm/N$ with a finite $N \trianglelefteq \Gm$. The image is abelian, hence $[G,G] \subseteq \ker \psi \bumpeq K$. Moreover, $(G/[G,G])/(K/[G,G]) \simeq G/K \simeq \Gm/N$ shows that the index of $[G,G]$ in $K$ is finite and thus $K \bumpeq [G,G]$ in~$G$. In~particular $\Psi(1,a) = \Set{ s \in \Gm \mid (1,a) \mathrel{\Psi} s } \subseteq N$ for all $(1,a) \in [G,G]$, i.e., $\Psi(1,a) \bumpeq 1$.
Using that $\Psi$ is a subgroup, we have $\Psi(s,a) \bumpeq \Psi(s,0) \cdot \Psi(1,a) \bumpeq \Psi(s,0)$. The restriction $\psi \restriction_H$ to $H = \Set{ (s,0) \in G } \simeq \Gm$ is, up to a finite quotient, a definable surjective endomorphism of~$\Gm$. By~\Cref{lemma: Def hom} it is hence of the form $\psi(s,0) = s^{n/p^m} N$ for some integers~$n \neq 0$ and $m \ge 0$. This~means $\Psi(s,a) \bumpeq s^{n/p^m}$.

The quasi-isomorphism $\Phi'$ acts on $\Gm$ like the Frobenius, up to a finite normal subgroup~$N' \trianglelefteq \Gm$. Thus, modulo the finite normal subgroup $N \cdot N' \trianglelefteq \Gm$ we have $(\Phi' \circ \Psi)(s,a) \bumpeq s^{np/p^m}$.
On~the other hand, let $\Phi(s,a) \bumpeq (t,d)$ and consider $(\Psi \circ \Phi)(s,a) \bumpeq \Psi(t,d) \bumpeq t^{n/p^m}$. The~assumption $\Psi \circ \Phi \bumpeq \Phi' \circ \Psi$ now reads $s^{np/p^m} \bumpeq t^{n/p^m}$. By~enlarging the finite normal subgroup of $\Gm$ implicit in this statement by $n$-th roots of unity and applying the Frobenius automorphism $m$ times, we obtain $t \bumpeq s^p$. Therefore, $\Phi(s,a) \bumpeq (s^p, d)$ for some~$d$ that depends~on~$(s,a)$.

Recall that $\tau\colon G \to \Aff$, $(s,a) \mapsto (s^k, a)$, has kernel $Z(G)$.
Let~$\pi\colon G \times G \to \Aff \times \Aff$ be the canonical quotient map modulo $Z = Z(G) \times Z(G)$. The subgroup $\pi(\Phi)$ is $\bumpeq$-equivalent to $\pi(\Phi)^0$. By~\Cref{lemma: Aut Aff}, this is the graph of a definable group automorphism $\varphi \in \Aut_{F}(\Aff)$, and \Cref{lemma: Aut G} guarantees that $\varphi(s,a) = (s^{p^r}, c a^{p^r} + b - s^{p^r} b)$ for $b \in \Ga$, $c \in \Gm$ and $r \in \mathbb{Z}$.
Consider the following relation $\ol{\varphi}$ on $G$ which is essentially the inverse image of $\varphi$ under $\tau \times \tau$:
\[
  (s,a) \mathrel{\ol{\varphi}} (t,d) \iff %
  \left(t^k = s^{kp^r} \;\land\;\, d = c a^{p^r} + b - s^{kp^r} b\right).
\]
This defines a subgroup of $G \times G$ with $\pi(\ol{\varphi}) = \varphi = \pi(\Phi)$ which shows that $\ol{\varphi} \bumpeq \Phi$. Hence,
\[
  (s^p, d) \bumpeq \Phi(s,a) \bumpeq \ol{\varphi}(s,a) \bumpeq (s^{p^r}, c a^{p^r} + b - s^{kp^r} b),
\]
so $r = 1$ and $\Phi(s,a) \bumpeq (s^p, c a^p + b - \theta_{s^p}(b))$ as claimed.
\end{proof}

\begin{lemma} \label{lemma: x centralizer}
The centralizer of $x$ in $F(x; \Frob)$ is the ring~$\FF_p(x)$ which is isomorphic to the ordinary function field~$\FF_p(t)$. Moreover, if $\varphi$ is conjugate to $x$, then its centralizer is $\FF_p(\varphi)$.
\end{lemma}

\begin{proof}
Let $S = \Laurent{F}{x; \Frob}$ be the skew Laurent series ring which contains $R = F(x; \Frob)$ as a subring \cite[Section~1.5]{CohnFIR}.
Let $C_S(x) = \Set{ f \in S \mid fx = xf }$ denote the centralizer of $x$ in~$S$. With the Frobenius canonically extended to $S$ by acting on the coefficients, it holds that $fx = xf$ if and only if $f = \Frob(f)$. By coefficient comparison every such $f$ has coefficients in $\FF_p$, thus $C_S(x) = \Laurent{\FF_p}{x; \Frob}$. From the definition of centralizer and \cite[Corollary~1.5.7]{CohnFIR} it then follows that $C_R(x) = C_S(x) \cap R = \FF_p(x; \Frob)$.
Finally, observe that the subring $C_R(x) = \FF_p(x; \Frob) \subseteq R$ is commutative and thus just the ordinary field of rational functions~$\FF_p(x)$. This observation also applies to $C_S(x)$ and shows that it equals the ordinary Laurent series ring $\Laurent{\FF_p}{x}$.

If $\varphi = y^{-1} x y$ then, using that $y$ commutes with every Laurent series coefficient from $\FF_p$,
\[
  C_S(\varphi) = y^{-1} \, C_S(x) \, y = y^{-1} \, \Laurent{\FF_p}{x} \, y = \Laurent{\FF_p}{y^{-1}xy} = \Laurent{\FF_p}{\varphi}.
\]
As before, cutting back to $R$ gives $C_R(\varphi) = C_S(\varphi) \cap R = \FF_p(\varphi)$.
\end{proof}

\begin{theorem}[Theorem 1.1(1)] \label{thm: main}
For each $p > 0$ the recognition problem for algebraic matroids in characteristic $p$ is undecidable.
\end{theorem}

\begin{proof}
We begin by recalling the framework for the reduction which is pictured in \Cref{fig: Aff conf}. The figure contains three group configurations: a gray one of rank two, a blue one of rank one, and a red one also of rank one. Each gray node is composed of two rank-one points which span~it.
Some lines are omitted from this figure for clarity. Their absence is indicated by three stubs coming out of a node. These lines make (superimposed) group configurations (as in \Cref{fig: triple compact}) as necessary so that \Cref{thm: quasi-aut} applies and recovers quasi-automorphism labels for all the points of interest. This enables von~Staudt constructions in the red and blue configurations of rank one (the dashed lines accommodate more nodes for these constructions).

We may force the red group to be~$\Ga$ by requiring $p = 0$ inside its quasi-automorphism skew~field. Similarly, we can force the blue group to be non-$\Ga$, i.e., either $\Gm$ or an elliptic curve, by imposing $p \neq 0$ in its skew field. By~\Cref{thm: Aff conf} the purple gadget then ensures that the blue group is~$\Gm$ and the gray rank-two group is $G = \Gm \ltimes_\theta \Ga$ with a non-trivial action $\theta_s(\alpha) = s^k \alpha$, $k \neq 0$.

In the $\Gm$ configuration we construct the quasi-automorphism $p \in \QQ = L_F(\Gm)$ using von~Staudt constructions. Suppose that the blue point $w'$ is labeled by the quasi-automorphism~$p$. By using the construction from \Cref{thm: Gm transfer} we obtain a quasi-automorphism $\Phi$ of $G$ which labels the whole grey point~$w$ and a quasi-epimorphism $\Psi \le G \times \Gm$ such that $\Psi \circ \Phi \bumpeq p \circ \Psi$. Since $p$ acts on $\Gm$ like the Frobenius, \Cref{prop: Gm to Aff quasi-aut} shows that $\Phi(s,\alpha) \bumpeq (s^p, q_1\alpha^p + q_2 - \theta_{s^p}(q_2))$ for some $q_1 \in \Gm$ and $q_2 \in \Ga$.

Denote $K = [G,G] = \{\, (1,\alpha) \mid \alpha \in \Ga \,\}\simeq\Ga$.
\Cref{thm: Aff conf} yields a quasi-isomorphism $\Upsilon \le K \times \Ga$ with the red $\Ga$ configuration. The induced quasi-automorphism $\Phi_K = \Phi \cap (K \times K)$ is given by $\Phi_K(1,\alpha) \bumpeq (1, q_1\alpha^p)$. Thus, by identifying $K$ with the $\Ga$ group in its second coordinate, $\Phi_K$ becomes $q_1x \in F(x; \Frob) = L_F(\Ga)$; similarly we view $\Upsilon$ as a quasi-automorphism of~$\Ga$. By~\Cref{thm: Ga transfer} the element $\varphi = \Upsilon (q_1x) \Upsilon^{-1} \in L_F(\Ga)$ is a quasi-automorphism of the red $\Ga$ configuration labeling the red point~$d$.

Since $F$ is algebraically closed, there exists $u \in F$ with $u^{p-1} = q_1$ and hence $\varphi = (\Upsilon u^{-1}) x (u\Upsilon^{-1})$ is conjugate to~$x$. By \Cref{lemma: x centralizer}, $C(\varphi) = \FF_p(\varphi)$ is an existentially definable copy of the (usual, commutative,) univariate function field inside $F(x; \Frob)$ which we have given together with its distinguished generator~$\varphi$. Using further von~Staudt constructions in the red $\Ga$ configuration, we can encode arbitrary existential sentences about $\FF_p(\varphi)$ in the algebraicity of a matroid in characteristic~$p$. Undecidability then follows from \Cref{thm: Pheidas--Videla}.
\end{proof}
\begin{corollary}[Theorem 1.1(2)]
    The recognition problem for algebraic matroids is undecidable.
\end{corollary}
\begin{proof}
    Recognition in fixed characteristic $p>0$ can be reduced to the general recognition problem: for each prime $p$ there exists an algebraic matroid $N_p$ which is algebraically realizable only in characteristic $p$; there is a family of such matroids $N_p$ for which the description is uniform in $p$ (\cite{LindstromChar}). Thus, $M$ is algebraic in characteristic $p$ if and only if $M\oplus N_p$ is algebraic.
\end{proof}

\section*{Acknowledgements}

\setlength{\intextsep}{5pt}%
\setlength{\columnsep}{5pt}%
\begin{wrapfigure}{R}{0.12\linewidth}
\vspace{-.5\baselineskip}%
\centering%
\href{https://doi.org/10.3030/101110545}{%
\includegraphics[width=0.9\linewidth]{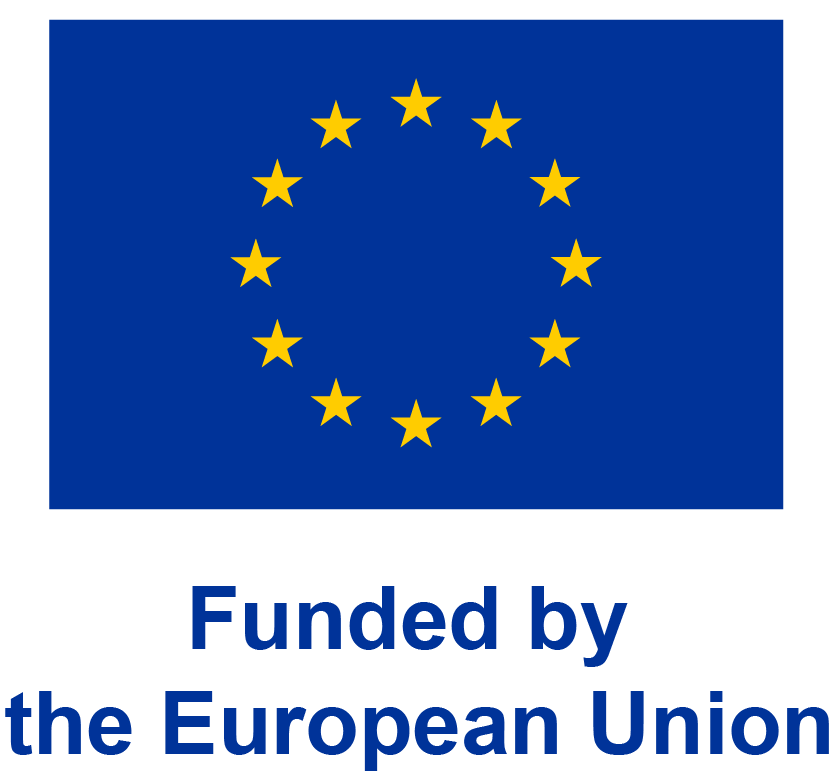}%
}
\end{wrapfigure}
T.B.~was funded by the European Union's Horizon 2020 research
and innovation programme under the Marie Skłodowska-Curie grant agreement
No.~101110545.

\hphantom{a} \\
\hphantom{a}

\bibliographystyle{preprint}
\bibliography{refs}

\end{document}